\documentclass[12pt,reqno]{amsart}

\usepackage{amsthm,amsmath,amsfonts,amssymb}
\usepackage{mathrsfs}
\usepackage{ascmac}
\usepackage{bm}
\usepackage[numbers]{natbib}
\usepackage[dvipdfmx]{graphicx}
\usepackage{graphicx,here,comment,subcaption}
\usepackage{amsaddr}
\usepackage{fancyhdr}
\usepackage{url}
\usepackage{booktabs}

\usepackage{fullpage}
\usepackage{hyperref}					
\hypersetup{colorlinks,					
linkcolor=blue,
	citecolor=blue}
\newtheorem{theorem}{Theorem}
\newtheorem{corollary}{Corollary}
\newtheorem{lemma}{Lemma}
\newtheorem{proposition}{Proposition}
\newtheorem{assumption}{Assumption}
\theoremstyle{definition}

\newtheorem{remark}{Remark}

\usepackage{amsaddr}

\title{Central limit theorem for the Sliced 1-Wasserstein distance and the max-Sliced 1-Wasserstein distance
$^*$}
\author{Xianliang Xu  \and Zhongyi Huang$^\dag$}
\address{Department of Mathematics, Tsinghua University, Beijing 100084, China.}
\email{xuxl19@mails.tsinghua.edu.cn; zhongyih@mail.tsinghua.edu.cn}
\thanks{$^\dag$ Corresponding author.}
\thanks{$^*$ This work was partially supported by the NSFC Projects No. 12025104, 11871298, 81930119.}
\begin{document}
\maketitle

\begin{abstract}
The Wasserstein distance has been an attractive tool in many fields. But due to its high computational complexity and the phenomenon of the curse of dimensionality in empirical estimation, various extensions of the Wasserstein distance have been proposed to overcome the shortcomings such as the Sliced Wasserstein distance. It enjoys a low computational cost and dimension-free sample complexity, but there are few distributional limit results of it. In this paper, we focus on Sliced 1-Wasserstein distance and its variant max-Sliced 1-Wasserstein distance. We utilize the central limit theorem in Banach space to derive the limit distribution for the Sliced 1-Wasserstein distance. Through viewing the empirical max-Sliced 1-Wasserstein distance as a supremum of an empirical process indexed by some function class, we prove that the function class is $P$-Donsker under mild moment assumption. Moreover, for computing Sliced $p$-Wasserstein distance based on Monte Carlo method, we explore that how many random projections that can make sure the error small in high probability. We also provide upper bound of the expected max-Sliced 1-Wasserstein between the true and the empirical probability measures under different conditions and the concentration inequalities for max-Sliced 1-Wasserstein distance are also presented. As applications of the theory, we utilize them for two-sample testing problems.  
\end{abstract}

\section{Introduction}
In many fields, it is fundamental to choose a proper distance to measure the discrepancy between two probability measures. Comparing to Kullback-Leibler(KL) divergence, $\chi^2$ statistic, the Wasserstein distance enjoys better geometric properties by considering the underlying geometry of the space. It has been employed successfully in statistics\cite{32,33,34,35}, machine learning\cite{36,37,38,39}, computer graphics\cite{40,41}. However, there are also many shortcomings that limit its practical applications. In the aspect of computation, the computational cost is of order $\mathcal{O}(n^3\log n)$ for discrete probability measure supported in points of size $n$. It becomes an obstacle of applying the Wasserstein distance in many machine learning tasks. Moreover, when true distributions are estimated from samples, the Wasserstein distance suffers from the curse of dimensionality which means that the convergence rate decays rapidly with dimension. Motivated by the fact that the Wasserstein distance has a closed form in one dimension, the Sliced Wasserstein distance has been proposed as an alternative to the original Wasserstein distance via averaging the Wasserstein distance between the random one-dimensional projections which follow the uniform distribution on the sphere\cite{42,43}. And its variant max-Sliced Wasserstein distance which is the maximum of the Wasserstein distance between the one-dimensional projections has also gained much attention.

Recently, the Sliced Wasserstein distance and the max-Sliced Wasserstein distance have been successfully applied in many tasks \cite{44,45,46,47,48}. Their sample complexity doesn't depend on the dimensionality\cite{49} and they have some similar properties as the Wasserstein distance. Several recent studies have explored the limit distribution for the Wasserstein distance, but there are few results for the Sliced Wasserstein distance and max-Sliced Wasserstein distance. In this study, we derive distribution limits for them. 

\subsection{Contributions}
Our contributions can be summarized as follows:
\begin{itemize}
\item We derive the central limit theorem result for the Sliced 1-Wasserstein distance based on the theory of probability in Banach space and the results of convergence rate in empirical estimation. We also explore the number of random projections that makes sure that Monte Carlo method has a small error in high probability in computing Sliced $p$-Wasserstein distance. 
\item We prove that the max-Sliced 1-Wasserstein distance between the true and the empirical distributions is a supremum of an empirical process indexed by some function class, which is $P$-Donsker when the ($4+\delta$)-th $(\delta > 0)$ moment of the true distribution is finite. And the limit distribution of max-Sliced 1-Wasserstein distance follows from the continuous mapping theorem.
\item We provide upper bound of expected max-Sliced 1-Wasserstein in empirical estimation and verify that it doesn't depend on the dimensionality. And the concentration inequalities of max-Sliced 1-Wasserstein distance are also presented.
\item We apply the Sliced 1-Wasserstein distance and the max-Sliced 1-Wasserstein in the problem of two-sample testing.
\end{itemize} 

\subsection{Related works}
In one dimension, \cite{1} derived the limit distribution of the empirical Wasserstein distance based on the central limit theorem in Banach space, more specifically in $L^1(\mathbb{R})$. \cite{17} explored the necessary and sufficient conditions for $Lip_1(\mathbb{R})$ to be $P$-Donsker, thus the limit distribution of empirical 1-Wasserstein distance follows from Kantorovich-Rubinstein(KR) duality and continuous mapping theorem. \cite{18} provided asymptotic distribution for a weighted version of empirical 2-Wasserstein distance. \cite{62} and \cite{60} also consider the case in one dimension under the $i.i.d.$ condition or non-$i.i.d.$ condition.

In general dimensions, combining Efron-Stein variance inequality and the stability results of optimal transportation potentials, \cite{4} showed that $
\sqrt{n}(W_2^2(P_n, Q)-\mathbb{E}W_2^2(P_n, Q)) $ converges to a zero-mean Gaussian distribution under that $P\neq Q$ and other moment assumptions. And \cite{19} sharpened their results in one dimension. Inspired by \cite{4}, \cite{57} obtained a central limit theorem of the form $\sqrt{n}(W_2^2(P,Q_n)-W_2^2(P,Q)) \xrightarrow{d} N(0, \sigma^2) $ under some regularity assumptions. When $P$ and $Q$ are supported in finite points, \cite{2} applied directional Hadamard differentiability theory to obtain limit distribution by dealing the Wasserstein distance between $P$ and $Q$ as a functional in finite dimensional space. And for $P$ and $Q$ supported in a countable set, \cite{3} derived limit distribution for empirical Wasserstein distance based on the sensitivity of optimal values of infinite dimensional mathematics programs and a delta method.

 For extensions of original Wasserstein distance like Smooth Wasserstein distance \cite{50} and entropic regularization of OT(EOT) \cite{51}, a central limit theorem for EOT was provided in \cite{20} for probability measures with finite support based on same technique in \cite{2} (see also \cite{21,22}). By bounding the entropy of the function class of Gaussian-smoothed Lipschitz functions and Theorem 1.1 in \cite{23}, \cite{5} showed that the function class is $P$-Donsker under some conditions. \cite{24} explored the limit distribution of projection-based Wasserstein distance based on the same argument of \cite{2} for distributions with finite support. And \cite{61} derived limit law of Sliced $p$-Wasserstein distance under that $P\neq Q$ for all $p>1$ by Hadamard differentiability theory, but it requires some regularity assumptions such as the absolutely continuity of all one-dimension projections. The most similar to our work is the concurrent work \cite{61}. For Sliced 1-Wasserstein distance, we obtain the necessary and sufficient conditions under which the central limit theorem holds and they present only the sufficient condition. And by the method in \cite{60} for 1-Wasserstein distance in one dimension, we derive the limit distribution for Sliced 1-Wasserstein distance under weaker conditions than \cite{61}. For max-Sliced 1-Wasserstein distance, our construction of the bracketing cover of the function class is different from their, our method is based on the bracketing number of 1-Lipschitz continuous functions vanishing at zero and their method is based on the covering number. The differences will be discussed in details in the paper.

\subsection{Notations}
$ \|\cdot \| $ denotes the Euclidean norm. The class of Borel probability measure on $\mathbb{R}^d$ is denoted by $\mathcal{P}(\mathbb{R}^d)$ and the  set of probability measure with finite moment of order $p$ is denoted by $\mathcal{P}_p(\mathbb{R}^d)$. Given a map $T:\mathbb{R}^d \rightarrow \mathbb{R}$ and a probability of $\mathcal{P}(\mathbb{R}^d)$, we write $T_{\#}P$ for the pushforward of $P$ under $T$, which means that for any Borel set $A \subset \mathbb{R}^d$ we have $T_{\#}P(A)=P(T^{-1}(A))$. For $A \in \mathbb{R}^{m \times n}$, define real-valued function $A^{*}: \mathbb{R}^m \rightarrow \mathbb{R}^n$ with $A^{*}(x) = A^T x $ for all $x \in \mathbb{R}^m$. Given $\mu \in \mathcal{P}(\mathbb{R}^d)$, define $M_p(\mu):= (\int_{\mathbb{R}^d} \|x\|^p d\mu(x) )^{\frac{1}{p}}$. We write $m$ for the Lebesgue measure in $\mathbb{R}$.

For non-empty set $\mathcal{T}$, let $l^{\infty}(\mathcal{T})$ denote the set of all bounded function $f: \mathcal{T} \rightarrow \mathbb{R}$ with respect to the uniform norm $\|f\|_{\mathcal{T}}:= \sup_{t \in \mathcal{T}}|f(t)|$. For the class of Lipschitz continuous real-valued function on $\mathbb{R}^d$ with Lipschitz constant at most 1, we denote it by $Lip_1(\mathbb{R}^d)$. And write $Lip_{1,0}(\mathbb{R}^d)$ for the functions that belong to $Lip_1(\mathbb{R}^d)$ and vanish at zero. We write $\mathcal{N}(\epsilon, \mathcal{F}, d)$ for the $\epsilon$-number of function class $\mathcal{F}$ w.r.t the metric $d$ and  $\mathcal{N}_{[\ ]}(\epsilon, \mathcal{F}, d)$ for the $\epsilon$-bracket number of $\mathcal{F}$. For brevity, we write $SW_p$, $MSW_p$ for the Sliced $p$-Wasserstein distance and max-Sliced $p$-Wasserstein distance respectively.

\section{Background}
\subsection{The Wasserstein Distance} 
The Wasserstein distance is related to the theory of optimal transport. Given two probability $P, Q \in \mathcal{P}(\mathbb{R}^d)$, the $p$-Wasserstein distance between them is defined as :
\[W_p(P,Q) = \left( \inf_{\pi \in \prod(P,Q)} \int_{\mathbb{R}^d\times \mathbb{R}^d} \|x-y\|^p d \pi(x,y)\right)^{\frac{1}{p}},\]
where $\prod(P,Q)$ is the set of so-called transport plans whose marginals coincide with $P$ and $Q$ respectively. When $d=1$, the Wasserstein distance has a closed form, that is
\[W_p(P,Q) = \left( \int_{0}^{1}|F^{-1}(x)-G^{-1}(x)|^p dx \right)^{\frac{1}{p}},\]
where $F$ and $G$ denote the cumulative distribution functions(CDF) of $P$ and $Q$ respectively and their pseudo-inverses are denoted by $F^{-1}$ and $G^{-1}$. Particularly, for $p=1$ and $d=1$, $W_1$ has the form as
\[W_1(P,Q)= \int_{-\infty}^{+\infty} |F(x)-G(x)|dx,\]
which can be regarded as the $L^1$ distance between the CDF of $P$ and $Q$. Besides, when $p=1$, KR duality shows that $W_1$ can be also represented as
\[W_1(P,Q) = \sup_{f \in Lip_1(\mathbb{R}^d)} \int_{\mathbb{R}^d} f d(P-Q).\]
It's obvious that $Lip_1(\mathbb{R}^d)$ can be replaced by $Lip_{1,0}(\mathbb{R}^d)$, since $P$ and $Q$ have the same total mass. Moreover, empirical $W_1$ can be seen as a supremum of an empirical process indexed by $Lip_{1,0}(\mathbb{R}^d)$. In one dimension, \cite{17} presented the necessary and sufficient conditions for  $Lip_{1,0}(\mathbb{R})$ to be $P$-Donsker. 

\subsection{Sliced 1-Wasserstein Distance and max-Sliced 1-Wasserstein Distance}
Trough slicing method, the Sliced $p$-Wasserstein distance is defined as
\[SW_p(P, Q) = \left( \int_{S^{d-1}} W_p^p(\theta_{\#}^{*}P, \theta_{\#}^{*}Q)d\sigma(\theta) \right)^{\frac{1}{p}}.\]
where $\sigma$ is the uniform distribution on $S^{d-1}$.

And the max-Sliced $p$-Wasserstein distance is defined as
\[MSW_p(P, Q) = \max_{\theta \in S^{d-1}} W_p(\theta_{\#}^{*}P, \theta_{\#}^{*}Q).\]

$SW_p$ has some similar properties as the Wasserstein distance and is equivalent to the Wasserstein distance when distributions are with compact support \cite{42}.

\section{Central Limit Theorem for the Sliced 1-Wasserstein Distance}
In this section, we explore the limit distribution of Sliced 1-Wasserstein distance between the empirical and the true probability measure. We derive similar results as Theorem 2.1 in \cite{1} which provided central limit theorem for the Wasserstein distance in one dimension. Based on the central limit theorem in Banach space, more specifically in $L^1(S^{d-1}\times \mathbb{R}, \sigma \times m)$, for brevity, we denote it by $L^1(S^{d-1}\times \mathbb{R})$. Notice that for $P,Q \in \mathcal{P}(\mathbb{R})$, $W_1(P,Q)$ is the $L^1$ distance between the CDF of $P$ and $Q$, thus for $\mu \in \mathcal{P}(\mathbb{R}^d)$, $SW_1(\mu_n,\mu)$ has the form as follows.
\[SW_1(\mu_n, \mu) = \int_{S^{d-1}} \int_{-\infty}^{+\infty}|F_n(\theta, t)-F(\theta, t)|dtd\sigma(\theta) ,\]
where $F_n(\theta, t) = \frac{1}{n} \sum\limits_{i=1}^n I_{\theta^T X_i\leq t}$, $F(\theta, t) = P(\theta^T X\leq t)$ and $X,X_1,\cdots, X_n$ are $i.i.d.$ random variables with values in $\mathbb{R}^d$ and common distribution $\mu$. In this Section, we define
\[ \mathcal{P}_{2,1}(\mathbb{R}^d)=\{\mu\in \mathcal{P}(\mathbb{R}^d) : \  \int_{0}^{\infty} \sqrt{ \mu(\|x\|>t) }dt<\infty \}   .\] 

\begin{theorem}
Let $X, X_i$,  $i \in \mathbb{N}$ be i.i.d. random variables with values in $\mathbb{R}^d$ and $X \sim  \mu$. Let
\[Y(\theta, t) := P(\theta^T X>t)-I_{\theta^T X>t}.\]
And $Y_i$'s definition is the same as $Y$, just replace the $X$ by $X_i$, then

(1) The processes $\sum\limits_{i=1}^n Y_i/\sqrt{n}= \sqrt{n}(F_n(\theta, t)- F(\theta, t))$ converge to $B_{\mu}$ in law in $L^1(S^{d-1}\times \mathbb{R})$ if and only if $\Lambda_{2,1}(X) < \infty$, where $B_{\mu}(\theta, t)$ is a centered Gaussian process with covariance function $\mathbb{E}[B_{\mu}(\theta_1, x_1)B_{\mu}(\theta_2, x_2)]=Cov(Y(\theta_1, x_1), Y(\theta_2, x_2))$ and $\Lambda_{2,1}(X) := \int_{0}^{+\infty} \sqrt{P(\|X\|>t)}dt.$

(2) The sequence 
\[\| \sum\limits_{i=1}^n \frac{Y_i}{\sqrt{n}} \|_{L^1(S^{d-1}\times \mathbb{R})} = \sqrt{n} \int_{S^{d-1}} \int_{-\infty}^{+\infty} |F_n(\theta, t)-F(\theta, t)|dtd\sigma(\theta) , n \in \mathbb{N},\]
is stochastically bounded if and only if $\Lambda_{2,1}(X) < \infty$.
\end{theorem}

By noticing the similarity between $W_1$ in one dimension and the $SW_1$, the method in \cite{60} can be applied to $SW_1$ naturally and leads to a weaker condition than Theorem 3 in \cite{61}.

\begin{proposition}
(1) For $\mu \in \mathcal{P}_{2,1}(\mathbb{R}^d)$ and $\nu \in \mathcal{P}(\mathbb{R}^d)$, let $F(\theta,t)=P(\theta^T X\leq t), G(\theta,t)=P(\theta^T Y\leq t)$ for $X\sim \mu, Y\sim \nu$, then 
\[\sqrt{n}H_n \xrightarrow{d} \int_{F>G}G_{\mu}dtd\sigma(\theta)-\int_{F<G}G_{\mu}dtd\sigma(\theta) +\int_{F=G} |G_{\mu}|dtd\sigma(\theta), \]
where 
$G_{\mu}$ is a centered Gaussian process with covariance function
\[cov(G_{\mu}(\theta_1, x_1), G_{\mu}(\theta_2, x_2))=\mu(\theta_1^T x\leq x_1, \theta_2^T x\leq x_2)-\mu(\theta_1^T x\leq x_1)\mu(\theta_2^T x\leq x_2).\]
and $H_n=\sqrt{n}\int_{S^{d-1}}\int_{-\infty}^{\infty} (|F_n(\theta,t)-G(\theta,t)|-|F(\theta,t)-G(\theta,t)|)ddtd\sigma(\theta)$. If in addition $\nu \in \mathcal{P}_1(\mathbb{R}^d)$, then $H_n=SW_1(\mu_n, \nu)-SW_1(\mu,\nu)$.

(2) Let $(X_i, Y_i)_{1\leq i \leq n}$ be a sequence of $i.i.d.$ random variables with values in $\mathbb{R}^{d}\times \mathbb{R}^{d} $ and assume that $\int_{S^{d-1}}\int_{-\infty}^{+\infty} \sqrt{var(I_{\{\theta^T X_1 \leq t\}} - I_{\{\theta^T Y_1 \leq t\}})}dtd\sigma(\theta)<\infty $.  Then  for probability measures $\mu_n= \sum\limits_{i=1}^n \delta_{X_i}/n $ and $\nu_n= \sum\limits_{i=1}^n \delta_{Y_i}/n $, we have
\[\sqrt{n}(SW_1(\mu_n, \nu_n)-SW_1(\mu,\nu))\xrightarrow{d} \int_{F>G}G^{'}dtd\sigma(\theta)-\int_{F<G}G^{'}dtd\sigma(\theta) +\int_{F=G} |G^{'}|dtd\sigma(\theta), \]
where $G^{'}$ is a centered Gaussian process with covariance function
\[cov(G^{'}(\theta_1, x_1), G^{'}(\theta_2, x_2))=cov(I_{\theta_1^T X_1\leq x_1}-I_{\theta_1^T Y_1\leq x_1}, I_{\theta_2^T X_1\leq x_2}-I_{\theta_2^T Y_1\leq x_2}). \]
\end{proposition}

\begin{remark}
\cite{61} has proved that the processes converge weakly in $L^1(S^{d-1}\times \mathbb{R})$ when $\mu \in \mathcal{P}_{2+\epsilon}(\mathbb{R}^d)$  for any $\epsilon>0$, but from Theorem 1, we can see that it is stronger than our condition since $ \mathcal{P}_{2+\epsilon}(\mathbb{R}^d) \subset \mathcal{P}_{2,1}(\mathbb{R}^d)$ holds for any $\epsilon>0$. Besides, they doesn't show the relation between the stochastical boundness of the sequence in Theorem 1(2) and the condition under which the central limit theorem holds. Comparing to \cite{61}, Theorem 3(i) in \cite{61} requires that $\mu \in \mathcal{P}_{2+\epsilon}(\mathbb{R}^d)$ and $\nu \in \mathcal{P}_{1}(\mathbb{R}^d)$, but the Proposition 1(1) just need that $\mu \in \mathcal{P}_{2,1}(\mathbb{R}^d)$ and $\nu \in \mathcal{P}(\mathbb{R}^d)$. Theorem 3(ii) in \cite{61} is for the condition that $X_i$ is independent of $Y_i$ and requires that $\mu,\nu \in \mathcal{P}_{2+\epsilon}(\mathbb{R}^d)$ for $X_1\sim \mu, Y_1\sim \nu$. Note that the fact that $\mu,\nu \in \mathcal{P}_{2+\epsilon}(\mathbb{R}^d)$ implies that $\int_{S^{d-1}}\int_{-\infty}^{+\infty} \sqrt{var(I_{\{\theta^T X_1 \leq t\}} - I_{\{\theta^T Y_1 \leq t\}})}dtd\sigma(\theta)<\infty $. Note also that if $X_i$ is independent of $Y_i$, $\int_{S^{d-1}}\int_{-\infty}^{+\infty} \sqrt{var(I_{\{\theta^T X_1 \leq t\}} - I_{\{\theta^T Y_1 \leq t\}})}dtd\sigma(\theta)<\infty $ is indeed equivalent to $\Lambda(X_1)<\infty$ and $\Lambda(Y_1)<\infty$.
\end{remark}

There are various results of convergence in expectation of Wasserstein distance in empirical estimation\cite{53,54,55,56}. Inspired by the results in one dimension\cite{13}, we have the results for $SW_1$ as follows.

\begin{proposition}
For any $\mu \in \mathcal{P}_1(\mathbb{R}^d)$, we have $\mathbb{E}[SW_1(\mu_n, \mu)] \to 0 , \ as \ n \to \infty.$	
\end{proposition}

And for the convergence rate, we have

\begin{corollary}
Given $\mu \in \mathcal{P}_1(\mathbb{R}^d)$, then $\mathbb{E}[SW_1(\mu_n, \mu)] = \mathcal{O}(\frac{1}{\sqrt{n}})$ if and only if $\Lambda_{2,1}(X) < \infty$ for $X \sim \mu$.
\end{corollary}

\begin{remark}
By Theorem 1 and Corollary 2, we obtain the equivalence among the central limit theorem for $SW_1$, stochastical boundness of the sequence in Theorem 1(2) and $\Lambda_{2,1}(X) < \infty$ for $X\sim \mu$. Moreover, $\frac{1}{\sqrt{n}}$ is the best rate for $\mathbb{E}[SW_1(\mu_n, \mu)]$ under that $\mu$ is not degenerate.
\end{remark}

\begin{proposition}
For $\mu \in \mathcal{P}_1(\mathbb{R}^d)$ that is not degenerate, we have 
\[\mathbb{E}[SW_1(\mu_n, \mu)] \geq c \frac{\mathbb{E}\|X-\mathbb{E}X\|}{\sqrt{n}} ,\]
where $X\sim \mu$ and $c$ is an absolute positive constant.
\end{proposition}

\section{The aspect of computing $SW_p$}
Despite that $SW_p$ has a low computational cost, it needs many random projections in high dimension to make the approximation error small when using Monte Carlo method. Inspired by the fact that one-dimensional projections of a high-dimensional random vector are approximately Gaussian under mild assumptions, \cite{6} developed a simple deterministic approximation for $SW_2$ without requiring random projections. But their method can't obtain an approximation with arbitrarily small error and just works for $SW_2$. By noticing the phenomenon of concentration of measure, we can give the number of random projections that can make error small in high probability when computing $SW_p$ based on Monte Carlo approximation.

\begin{proposition}
For fixed $\epsilon$ and $\delta$, when the number of random projections $n$ satisfies that $ n \geq\frac{2L^2}{(d-1) \epsilon^2 } \log\frac{2}{\delta}$, we have
\[P\left(|SW_p^p(\mu, \nu) - {SW}_{p,n}^p(\mu,\nu)| \geq \epsilon \right) \leq \delta,\]
where ${SW}_{p,n}^p(\mu, \nu) = \frac{1}{n}\sum\limits_{i=1}^n W_p^p({\theta_i}_{\#}^{*}\mu, {\theta_i}_{\#}^{*}\nu)$, $\{\theta_i\}_{i=1}^n$ i.i.d. from $\sigma$ and $L=pW_p^{p-1}(\mu, \nu)(M_p(\mu)+M_p(\nu))$. 
\end{proposition}

Notice that $|x-y|^p \leq |x^p-y^p|$ holds for any positive number $x,y$ and $p \geq 1$. Then by Proposition 2, we obtain the result for $SW_p$.  

\begin{corollary}
For fixed $\epsilon$ and $\delta$, when the number of random projections $n$ satisfies that $ n \geq \frac{2L^2}{(d-1) \epsilon^{2p} } \log\frac{2}{\delta}$, we have
\[P\left(|SW_p(\mu, \nu) - {SW}_{p,n}(\mu,\nu)| \geq \epsilon \right) \leq \delta,\]
where ${SW}_{p,n}(\mu, \nu) = \left({SW}_{p,n}^p(\mu, \nu) \right)^{1/p}$.
\end{corollary}

From the lower bound of $n$ in Corollary 3, we can see that for $p \neq 1$, there should be a proper prior estimation of $W_p(\mu, \nu)$ first in practice. Thus there is an obstacle between theory and application. But for $p=1$, there is no such problem and $n$ may be independent of $d$ under some condition.

\begin{corollary}
	 For $n \geq \frac{4(\Delta_{\mu}+\Delta_{\nu})^2}{ \epsilon^{2} }\log\frac{2}{\delta}$, we have 
\[P\left(|SW_1(\mu, \nu) - {SW}_{1,n}(\mu,\nu)| \geq \epsilon \right) \leq \delta, \]
where $\Delta_{\mu}=\max_{i}M_2((\pi_i)_{\#}\mu)$ and $\Delta_{\nu}=\max_{i}M_2((\pi_i)_{\#}\nu)$.
\end{corollary}

Thus if there is a dimension independent upper bound for $\max_{i}M_2((\pi_i)_{\#}\mu)$ and $\max_{i}M_2((\pi_i)_{\#}\nu)$, then $n$ is independent of $d$. For example, let $(X_i)_{i\in\mathbb{N}}$ be a stationary sequence of one-dimensional random variables, i.e. $X_i$ and $X_j$ have the same distribution for any $i,j \in\mathbb{N}$ and $(Y_i)_{i\in\mathbb{N}}$ be another stationary sequence. Denote by $\mu_d$ and $\nu_d$ the distribution of $(X_1,\cdots, X_d)$ and $(Y_1,\cdots, Y_d)$ respectively. Then by Corollary 3, $n$ is independent of $d$.

For $p\neq 1$, the problem can be handled by replace $SW_p$ with a variant of it, which is defined as 
$$\widehat{SW_p}(\mu, \nu) = \int_{S^{d-1}} W_p({\theta}_{\#}^{*}\mu, {\theta}_{\#}^{*}\mu) d\sigma(\theta).$$

It's easy to check that $\widehat{SW_p}$ is indeed a distance. Then, based on the same technique, we have: 
\begin{corollary}
For fixed $\epsilon$ and $\delta$, when $n \geq\frac{2\tilde{L}^2}{(d-1) \epsilon^2 } \log\frac{2}{\delta}$, we have
\[P(|\widehat{SW}_p(\mu, \nu)-\widehat{SW}_{p,n}(\mu, \nu)|\geq \epsilon) \leq \delta,\]
where $\widehat{SW}_{p,n}(\mu, \nu) =\frac{1}{n}\sum\limits_{i=1}^n W_p({\theta_i}_{\#}^{*}\mu, {\theta_i}_{\#}^{*}\nu)$, $\{\theta_i\}_{i=1}^n$ i.i.d. from $\sigma$ and $\tilde{L}=M_p(\mu)+M_p(\nu)$. 
\end{corollary}

\begin{remark} 
With no priori information about $\mu$ and $\nu$, it may be reasonable to choose random projections uniformly distributed on the unit sphere. There are also variant potential distribution for random projections deserved to be explored. For example, \cite{6} replaced the uniform distribution by a zero-mean Gaussian distribution with covariance matrix $(1/d)I_d$ and showed that there are equivalent up to a constant that only depends the dimension $d$ and the order $p$. Here, we present the result.
\[\widetilde{SW}_p^p(\mu, \nu):= \int_{R^d} W_p^p({\theta}_{\#}^{*}\mu, {\theta}_{\#}^{*}\nu)d\gamma_d(\theta) =(2/d)^{1/2}\{\Gamma(d/2+p/2)/\Gamma(d/2)\}^{1/p}SW_p^p(\mu, \nu) ,\]
where $p \in [1,+\infty)$, $\Gamma$ is the Gamma function and $\gamma_d$ is the zero-mean Gaussian distribution with covariance matrix $(1/d)I_d$. In particular, when $p=2$, $ \widetilde{SW}_2^2(\mu, \nu)=SW_2^2(\mu, \nu)$. Let $c_{p,d}$ denote the constant $(2/d)^{1/2}\{\Gamma(d/2+p/2)/\Gamma(d/2)\}^{1/p}$, i.e. $\widetilde{SW}_p^p(\mu, \nu)=c_{p,d}SW_p^p(\mu, \nu)$. 
	
Under the observation that the Lipschitz function of Gaussian variables is sub-Gaussian (see Theorem 2.26 in \cite{14}), we can obtain similar results for $\widetilde{SW}_p$. 
\end{remark}

\begin{corollary}
For fixed $\epsilon$ and $\delta$, when $n \geq \frac{2L^2}{d^p \epsilon^2} \log{\frac{2}{\delta}}$, we have
\[P(|\widetilde{SW}_p^p(\mu, \nu)-\frac{1}{d^{p/2 }}\overline{SW}_{p,n}^n | \geq \epsilon) \leq \delta,\]
where $\overline{SW}_{p,n}^p(\mu, \nu) =\frac{1}{n}\sum\limits_{i=1}^n W_p({\theta_i}_{\#}^{*}\mu, {\theta_i}_{\#}^{*}\nu)$, $\{\theta_i\}_{i=1}^n$ i.i.d. from the zero-mean Gaussian distribution with covariance matrix $I_d$ and $L$ is the same as in Proposition 2. 
\end{corollary}

Through the equivalence between $\widetilde{SW}_p^p$ and ${SW}_p^p$, we can estimate ${SW}_p^p$  from the estimation of $\widetilde{SW}_p^p$.

\begin{corollary}
For fixed $\epsilon$ and $\delta$, when $n \geq \frac{2L^2}{d^p  c_{p,d}^2\epsilon^2} \log{\frac{2}{\delta}}$, we have
\[P(|{SW}_p^p(\mu, \nu)-\frac{1}{d^{p/2} c_{p,d}}\overline{SW}_{p,n}^n | \geq \epsilon) \leq \delta.\]
\end{corollary}

\section{Central Limit Theorem for Max-Sliced 1-Wasserstein Distance}
\subsection{Limit Distribution for $MSW_1$}
In this section, we characterize the limit distribution of $\sqrt{n}MSW_1(P_n,P)$ via the theory of empirical process. First, we show that $MSW_1(P_n,P)$ can be translated to the supremum of an empirical process through KR duality.

\begin{lemma} 
Given $P \in \mathcal{P}(\mathbb{R}^d)$, we have $MSW_1(P_n, P) = \sup\limits_{f \in \mathcal{F}}(P_n-P)f$, where
\[\mathcal{F}:= \{f :\mathbb{R}^d\to \mathbb{R} \ | \ f(x)=g(\theta^T x), g \in Lip_{1,0}(\mathbb{R}), \theta \in S^{d-1} \}.\]
\end{lemma}
By bounding the bracketing entropy integral of $\mathcal{F}$ with respect to $P$ which is defined as
\[J_{[\ ]}(\delta, \mathcal{F}, L_2(P))=\int_{0}^{\delta} \sqrt{\log\mathcal{N}_{[\ ]}(\epsilon, \mathcal{F}, L_2(P))}d\epsilon,\]
we have result as follows. 

\begin{theorem}
When the ($4+\delta$)-th ($\delta>0$) moment of $P$ is finite, $\mathcal{F}$ is $P$-Donsker. Then from continuous mapping theorem, we have
\[\sqrt{n} MSW_1(P_n, P) \xrightarrow{d} \|G_P\|_{\mathcal{F}},\]
where $G_P$ is a tight version of a zero-mean Gaussian process on $\mathcal{F}$ with covariance function $Cov(G_Pf,G_Pg) = P(fg)-(Pf)(Pg)$.
\end{theorem}

The technique in the proof of Theorem 2 can also be applied to other integral probability metrics (IPM) defined as 
\[\gamma_{\mathcal{F}}(P,Q):= \sup_{f \in \mathcal{F}} |\int f dP-\int fdQ|,\] 
and the max-Sliced version of the IPM is defined as
\[\max-\gamma_{\mathcal{F}}(P,Q):= \max_{\theta \in S^{d-1}} \sup_{f \in \mathcal{F}} |\int f d {\theta}_{\#}^{*}P-\int fd{\theta}_{\#}^{*}Q|.\]

Similar to Lemma 1, we have
\[ \max-\gamma_{\mathcal{F}}(P,Q) = \sup_{f \in \widehat{\mathcal{F}}} | \int f dP-\int f dQ|,\]
where $\widehat{\mathcal{F}}:= \{ f :\mathbb{R}^d \to \mathbb{R} \ | \  f(x)=g(\theta^T x), g \in \mathcal{F}, \theta \in S^{d-1} \}$. \\

For function class satisfied the following assumptions :
\begin{assumption} 
$\mathcal{F} \subset Lip_L(\mathbb{R})$, for some positive constant $L$.
\end{assumption}
\begin{assumption}
For any positive constant $M \in \mathbb{R}$, the bracketing entropy of $\mathcal{F}| _{[0, M]}$ satisfies that $\log\mathcal{N}_{[\ ]}(\epsilon, \mathcal{F} |_{[0, M]}, \|\cdot \|_{\infty}) = O(\frac{M^{\alpha}}{{\epsilon}^{\beta}})$ for some  constants $\alpha > 0$ and $0< \beta < 2$, where $\mathcal{F} |_{[0, M]}:=\{f |_{[0, M]} \ | \ f \in \mathcal{F} \}$.  
\end{assumption}

Then we have the following result for IPMs.
\begin{corollary} 
If the $(2+\frac{2\alpha}{2-\beta})$-th moment of $P$ is finite, $\widehat{\mathcal{F}}$ is $P$-Donsker and the continuous mapping theorem implies 
\[\sqrt{n}\max-\gamma_{\mathcal{F}}(P_n,P) \xrightarrow{d} \|G_P\|_{\widehat{\mathcal{F}}},\]
where $G_P$ is a tight version of a zero-mean Gaussian process on $\widehat{\mathcal{F}}$ with covariance function $Cov(G_Pf,G_Pg) = P(fg)-(Pf)(Pg)$.
\end{corollary}

\section{Statistical analysis for $MSW_1$}
\subsection{Upper bound in empirical estimation}
It's obvious that $SW_1$ enjoys dimension-free sample complexity, i.e. the convergence rate under empirical estimation is independent of dimensionality. But it's not clear for $MSW_1$. \cite{66} has showed the upper bound when the probability measure satisfies the transportation cost inequality. Also \cite{52} has explored the convergence rate for projection based Wasserstein distance under strong assumptions like satisfying the Bernstein tail condition or projection Poincare inequality and the results in \cite{52} imply that $MSW_1$ also has sample complexity independent of dimensionality. In the next, we use chaining method to obtain the convergence rate of $MSW_1$ under different conditions. For $\alpha>0$, let $\|\xi\|_{\psi_{\alpha}}:= \inf\{C>0: \mathbb{E}[e^{\|\xi\|/C}] \leq 2\}$ be the Orlicz $\psi_{\alpha}$-(quasi)norm for real-valued random variable $\xi$.

\begin{theorem}
Suppose that one of the conditions holds: (1) $\mu \in \mathcal{P}_p(\mathbb{R}^d)$ for $1<p<2$; (2) $\mu \in \mathcal{P}_p(\mathbb{R}^d)$ for $2\leq p \leq 4$; (3) $\mu \in \mathcal{P}_p(\mathbb{R}^d)$ for $p>4$, then
\begin{equation*}
\mathbb{E}[MSW_1(\mu_n, \mu)]=\left\{
\begin{aligned}
&\mathcal{O}(n^{-\frac{p-1}{2p}}), &under \  (1), \\
&\mathcal{O}(n^{-\frac{p-1}{2p-1}}), &under \  (2),\\
&\mathcal{O}(n^{-\frac{1}{2}}), &under\  (3).\\
\end{aligned}
\right.
\end{equation*}
\end{theorem}

\begin{remark}
Note that in \cite{66} the convergence rate is $\mathcal{O}(\sqrt{\frac{\log n}{n}})$ under the condition that $\mu$ satisfies the $T_1(\sigma^2)$ inequality which is equivalent to that $\|\|X\|\|_{\psi_{2}}<\infty$. In condition (3), our condition is weaker and upper bound is tigher. Notice that under the condition (3), we can conclude from Proposition 3 that the upper bound is also sharp.

\end{remark}

\subsection{Concentration Inequalities}
In this section, we consider the concentration inequalities for $MSW_1$ and $SW_1$ under different conditions which is similar with Proposition 1 in \cite{5}. Trough treating $MSW_1$ between the empirical probability and the true probability as the supremum of an empirical process, the results follow directly from \cite{26}.

\begin{proposition} 
Let $X \sim P$ with $supp(P) = \mathcal{X}$ and suppose one of the following conditions holds :(i) $\mathcal{X}$ is compact; (ii) $\| \|X\| \|_{\psi_\alpha} < \infty $, for some $ \alpha \in (0, 1]$. Then for all $\eta \in (0, 1)$, there exists constant $C:=C_{\eta}$ such that for all $t >0$:
\[P(MSW_1(\mu_n, \mu) \geq (1+\eta)\mathbb{E}[ MSW_1(\mu_n, \mu)+t)]) \leq \exp(-\frac{nt^2}{2(1+\delta)\mathbb{E}\|x\|^2}) + \alpha(n,t),\]
and
\[P(MSW_1(\mu_n, \mu) \leq (1-\eta)\mathbb{E}[ MSW_1(\mu_n, \mu)-t)] )\leq \exp(-\frac{nt^2}{2(1+\delta)\mathbb{E}\|x\|^2}) + \alpha(n,t),\]
where
\begin{equation*}
\alpha(n,t) = \left \{
\begin{aligned}
&\exp\left(-\frac{nt}{C diam(\mathcal{X})} \right), & under \ (i).\\
&3\exp 
\left(-  \left(\frac{n t}{C\| \|X\|+\mathbb{E}\|X\|\|_{\psi_\alpha} }\right)^{\alpha}   \right), &under \ (ii) .
\end{aligned}
\right.
\end{equation*}
\end{proposition}

If $\mu$ is induced by a $\sigma^2$-Subgaussian random vector $X$ with value in $\mathbb{R}^d$, i.e. for all $\alpha \in \mathbb{R}^d$, 
\[ \mathbb{E}[\exp(\alpha^T (X-\mathbb{E}X))] \leq \exp(\|\alpha\|^2\sigma^2/2 )\]
then we have 
\[P(|MSW_1(\mu_n, \mu)-\mathbb{E}MSW_1(\mu_n, \mu)|\geq t) \leq 2e^{-\frac{nt^2}{32d\sigma^2}},\]

and 

\[P(|SW_1(\mu_n, \mu) -\mathbb{E}SW_1(\mu_n,\mu)|\geq t) \leq 2e^{-\frac{nt^2}{4\sigma^2}}.\]

\subsection{Two-Sample Settings}
Let $X_1, \cdots, X_m$ and $Y_1, \cdots, Y_n$ be independent random samples from distributions $P$ and $Q$ respectively. It's fundamental in statistics to test the null hypothesis $H_0: P=Q$ versus the alternative $H_1:P \neq Q$. By considering $MSW_1$, we define the statistic
\[W_{m,n}:= \sqrt{\frac{mn}{N}}MSW_1(P_m,Q_n)= \sqrt{\frac{mn}{N}}\|P_m-Q_n\|_{\mathcal{F}},\]
where $\mathcal{F}$ is the function class in Lemma 1 and $N=m+n$. Following by chapter 3.7 in \cite{15}, let  
$(Z_{N1}, \cdots, Z_{NN})=(X_1, \cdots, X_m, Y_1, \cdots, Y_n)$ and $\lambda=m/N$, consider the pooled empirical measure 
\[H_N = \frac{1}{N} \sum\limits_{i=1}^{N}\delta_{Z_{Ni}}=\lambda_N P_m +(1-\lambda_N)Q_n.\]
Let $Z_1, \cdots, Z_N$ be a bootstrap sample from $H_N$ and define
\[P_m^B:= \frac{1}{m}\sum\limits_{i=1}^m \delta_{Z_i}, \ Q_n^B:= \frac{1}{n}\sum\limits_{i=m+1}^N \delta_{Z_i}.\]
Let $W_{m,n}^B=\sqrt{\frac{mn}{N}}MSW_1(P_m^B,Q_n^B)$, we have the following corollary.

\begin{corollary}
Let $P,Q \in \mathcal{P}(\mathbb{R}^d)$. Assume that the $(4+\delta)$-th moment of $P$ and $Q$ are finite. Then with the same notation as above, we have
	 
(i) If $P=Q$, then $W_{m,m}\xrightarrow{d} G_P$, and $W_{m,n} \xrightarrow{P} \infty$, otherwise. 

(ii) $W_{m,n}^B \xrightarrow{d} G_H$ given almost every sequence $X_1, \cdots, X_m, Y_1, \cdots, Y_n$ where $G_H$ is a tight Brownian bridge process corresponding the measure $H=\lambda P + (1-\lambda) Q$. 
\end{corollary}

For $P \neq Q$, we consider the following statistic.
\[\overline{W}_{m,n}:= \sqrt{\frac{mn}{m+n}} \left( MSW_1(P_m, Q_n)-MSW_1(P,Q) \right).\]

Following from a direct application of Theorem 6.1 in \cite{28}, we derive limit distribution for $
\overline{W}_{m,n}$ which is similar to Corollary 4 in \cite{5}.

\begin{corollary}
Under the assumptions of Corollary 10, let
\[d(f,g) := \sqrt{\mathbb{E}_P(f-g)^2}+\sqrt{\mathbb{E}_Q(f-g)^2},f,g \in \mathcal{F}.\]
Then $(\mathcal{F},d)$ is a totally bounded pseudometric space. Moreover, we have 
\[\overline{W_{m,n}} \to \sup_{\overline{M}^+(D,d)} G,\]
where $G:= \sqrt{1-\lambda}G_P-\sqrt{\lambda}G_Q$,
\[\widetilde{M}^+(D,d) := \{f\in (\overline{\mathcal{F}}, d): \mathbb{E}_P(f)-\mathbb{E}_Q(f) = MSW_1(P,Q)\},\]
and $\overline{\mathcal{F}}$ is the completion of $\mathcal{F}$ with respect to $d$.
\end{corollary}

\section{Application}
\subsection{Two-Sample Testing}
The Problem of two-sample testing is a fundamental topic in statistics. Specifically, for $\mu, \nu \in \mathcal{P}(\mathbb{R}^d)$, let $X_1,\cdots, X_m \sim \mu$ and $Y_1, \cdots, Y_n\sim \nu$ be independent samples from $\mu$ and $\nu$ respectively. Based on these samples, we wish to know whether $\mu=\nu$, that is, decide whether to accept the null hypothesis $H_0 \ : \ \mu=\nu$ or alternative hypothesis $H_1 \ : \ \mu \neq \nu$. In this section, we consider the non-parametric two-sample testing and utilize $SW_1$ and $MSW_1$ to construct test statistic. 

First, we state the definition of asymptotically level. For a sequence of two-sample tests $T_{m,n}$ defined as
\begin{equation*}
T_{m,n}(X_1,\cdots, X_m, Y_1, \cdots, Y_n) =\left\{
\begin{aligned}
H_1&, &if \ D_{m,n}>c_{m,n}, \\
H_0&,  &otherwise,
\end{aligned}
\right.
\end{equation*} 
where $D_{m,n}$ is a sequence of statistics based on samples $(X_1,\cdots, X_m, Y_1, \cdots, Y_n)$. We say $T_{m,n}$ is asymptotically level $\alpha\in [0,1]$ if $\lim\sup_{m,n\to \infty}P(T_{m,n}>c_{m,n}) \leq \alpha$ under $H_0$ and $T_{m,n}$ is consistent at asymptotic level $\alpha$ if it is asymptotic level $\alpha$ and $\lim\inf_{n \to \infty}P(T_{m,n}>c_{m,n})=1$ under $H_1$. 

\begin{proposition}
Under the setting of Section 6.3, consider the problem of testing $H_0 \ : \ \mu=\nu$ versus  $H_1 \ : \ \mu \neq \nu$. Let 
\begin{equation*}
D_{m,n}=\left\{
\begin{aligned}
&\overline{W}_{m,n}= \sqrt{\frac{mn}{N}} \|P_m-Q_n\|_{L^1(S^{d-1}\times \mathbb{R})}, &(1), \\
&W_{m,n}=\sqrt{\frac{mn}{N}}  \|P_m-Q_n\|_{\mathcal{F}}, &(2),
\end{aligned}
\right.
\end{equation*}	
and $T_{m,n}$ is defined as above, where under (1) $c_{m,n}$ is the $\alpha$-quantile of $\overline{W}_{m,n}^B:=\sqrt{\frac{mn}{N}} \|P_m^B-Q_n^B \|_{L^1(S^{d-1}\times \mathbb{R})}$
and under (2) $c_{m,n}$ is the $\alpha$-quantile of $W_{m,n}^B$. Then, this test is consistent at asymptotic level $\alpha$ if under (1) $\Lambda_{2,1}(X)<\infty$ and $ \Lambda_{2,1}(Y)<\infty$ for $X\sim \mu$ and $Y\sim \nu$; under (2) the $(4+\delta)$-th ($\delta>0$) moment of $\mu$ and $\nu$ are finite.
 
\end{proposition}

\section{Concluding Remarks}
In this work, we characterize the limit distribution of empirical $SW_1$ and $MSW_1$. Furthermore, for the computation of $SW_p$ based on Monte Carlo method, we show the number of random projections that can make sure the error small in high probability and propose a variant of $SW_p$ which can be estimated easily. Also the results of convergence for $SW_1$ are presented. For $MSW_1$, we present concentration inequalities results and upper bound in expectation in empirical estimation. 

In the future, we want to explore the central limit theorem for $SW_p$. For $SW_p$ and $MSW_p$, other possible direction include applications in statistics and developing efficient algorithms, which is a core problem in practice.

\section{Appendix}
\subsection{Proof of Theorem 1(i)} 
\begin{proof}

Notice that $Y(\theta, t) \in L^1(S^{d-1}\times \mathbb{R})$ a.e. as long as $\mathbb{E}\|X\| < \infty$, since
\begin{align*}
\|Y\|_{L^1} &= 
\int_{S^{d-1}} \int_{-\infty}^{+\infty} |P(\theta^T X>t)-I_{\theta^T X>t}|dtd\sigma(\theta) \\
&=\int_{S^{d-1}}\int_{-\infty}^{\theta^T X} P(\theta^T X \leq t)dtd\sigma(\theta)+
\int_{S^{d-1}}\int_{\theta^T X}^{+\infty} P(\theta^T X > t)dtd\sigma(\theta) \\
&\leq \int_{S^{d-1}}\int_{-\infty}^{|\theta^T X|} P(\theta^T X \leq t)dtd\sigma(\theta)+
\int_{S^{d-1}}\int_{-|\theta^T X|}^{+\infty} P(\theta^T X > t)dtd\sigma(\theta) \\
&\leq \int_{S^{d-1}}\int_{-\infty}^{-|\theta^T X|} P(\theta^T X \leq t)dtd\sigma(\theta)+
\int_{S^{d-1}}\int_{|\theta^T X|}^{+\infty} P(\theta^T X > t)dtd\sigma(\theta)+4|\theta^T X| \\
&= \int_{S^{d-1}}\int_{|\theta^T X|}^{+\infty} P(\theta^T X \leq -t)+P(\theta^T X > t)dtd\sigma(\theta)+4\|X\| \\
&\leq \int_{S^{d-1}}\int_{0}^{+\infty} P(|\theta^T X| \geq t)dtd\sigma(\theta)+4\|X\| \\
&\leq \mathbb{E} \|X\| +4\|X\| .
\end{align*}

 Thus following from the central limit theorem in Banach space (Theorem 10.10 in \cite{11}, see also Theorem 3 in \cite{9} which is more specific), we have that $\sum\limits_{i=1}^{n}Y_i/ \sqrt{n}$ converges weakly in $L^1(S^{d-1}\times \mathbb{R})$ if and only if $\int_{S^{d-1}} \int_{-\infty}^{+\infty} \left(\mathbb{E}[|Y(\theta,t)|^2]  \right)^{1/2}
dt d\sigma(\theta) < \infty$ and $P(\|Y\|_{L^1} >t) = o(t^{-2})$. 

Notice that $\|Y\|_{L^1} \leq \mathbb{E} \|X\| +4\|X\|$, so if $\mathbb{E}\|X\|^2 < \infty $ , then from Markov inequality, $P(\|Y\|_{L^1} >t) = o(t^{-2})$ holds.

Since $\mathbb{E}[|Y(\theta,s)|^2] = P(\theta^T X > t)(1-P(\theta^T X > t) ) = F(\theta,t)(1-F(\theta,t))$, we have
\[\int_{S^{d-1}} \int_{-\infty}^{+\infty} \left(\mathbb{E}[|Y(\theta,t)|^2]  \right)^{1/2}
dt d\sigma(\theta) < \infty \Leftrightarrow  \int_{S^{d-1}} \int_{-\infty}^{+\infty} \sqrt{F(\theta,t)(1-F(\theta,t))}dt d\sigma(\theta) < \infty.\]
We divide the proof into three steps .

\textbf{Step 1:} In the following, we prove that
\[ \int_{S^{d-1}} \int_{-\infty}^{+\infty} \sqrt{F(\theta,t)(1-F(\theta,t))}dtd\sigma(\theta) < \infty \Leftrightarrow  \int_{S^{d-1}} \int_{0}^{+\infty} \sqrt{P(|\theta^T X|>t)}dtd\sigma(\theta) < \infty.\]
We first prove " $\Leftarrow$ ":
\begin{align*}
\int_{-\infty}^{+\infty} \sqrt{F(\theta,t)(1-F(\theta,t))}dt &= 
\int_{-\infty}^{+\infty}\sqrt{P(\theta^T X \leq t)P(\theta^T X>t) }dt \\
&= \int_{0}^{+\infty}  \sqrt{P(\theta^T X \leq t)P(\theta^T X>t) }+\sqrt{P(\theta^T X \leq -t)P(\theta^T X>-t) }dt \\
& \leq \int_{0}^{+\infty} \sqrt{P(\theta^T X>t) } + \sqrt{P(\theta^T X \leq -t)}dt \\
& \leq \int_{0}^{+\infty} \sqrt{2} \sqrt{P(\theta^T X\geq t)+P(\theta^T X\leq -t)}dt \\
&=\int_{0}^{+\infty} \sqrt{2} \sqrt{P(|\theta^T X| \geq t)}dt \\
&  = \int_{-1}^{+\infty}\sqrt{2} \sqrt{P(|\theta^T X| \geq t+1)}dt \\
&\leq \int_{-1}^{+\infty}\sqrt{2} \sqrt{P(|\theta^T X| > t)}dt \\
&= \sqrt{2}+\int_{0}^{+\infty}\sqrt{2} \sqrt{P(|\theta^T X| > t)}dt ,
\end{align*}
where the second inequality follows from $\sqrt{a}+\sqrt{b}\leq \sqrt{2}\sqrt{a+b}$ for $a\geq 0, b\geq 0$. Thus 
\[ \int_{S^{d-1}} \int_{0}^{+\infty} \sqrt{P(|\theta^T X|>t)}dtd\sigma(\theta) < \infty \Rightarrow
\int_{S^{d-1}} \int_{-\infty}^{+\infty} \sqrt{F(\theta,t)(1-F(\theta,t))} dt d\sigma(\theta)< \infty.\]

Next, we prove "$\Rightarrow$": 

As $P(|\theta^T X|>t) \leq \frac{\mathbb{E}|\theta^T X|}{t}\leq \frac{\mathbb{E}\|X\|}{t}$, we have that for any $\epsilon > 0$, there exists a constant $t_0 > 0$ such that for any $\theta$ and $t \geq t_0$, $P(|\theta^T X| > t) \leq \epsilon$ holds. It leads to 
\begin{align*}
&\int_{-\infty}^{+\infty} \sqrt{F(\theta,t)(1-F(\theta,t))}dt  \\
&= \int_{0}^{+\infty}  \sqrt{P(\theta^T X \leq t)P(\theta^T X>t) }+\sqrt{P(\theta^T X \leq -t)P(\theta^T X>-t) }dt \\
&\geq\int_{t_0}^{+\infty}  (\sqrt{P(\theta^T X \leq t)P(\theta^T X>t) }dt+ \int_{t_0}^{+\infty} \sqrt{P(\theta^T X \leq -t)P(\theta^T X>-t) })dt .
\end{align*}
For the first integral, we have
\begin{align*}
\int_{t_0}^{+\infty}  \sqrt{P(\theta^T X \leq t)P(\theta^T X>t) }  dt &= 
\int_{t_0}^{+\infty}  \sqrt{P(\theta^T X > t)(1-P(\theta^T X > t)) }  dt \\
&\geq \sqrt{1-\epsilon} \int_{t_0}^{+\infty}  \sqrt{P(\theta^T X > t) }  dt .
\end{align*}
The inequality follows from that when $t \geq t_0$, $P(\theta^T X > t) \leq P(|\theta^T X| > t) \leq \epsilon$.

Similarly, we have
\[ \int_{t_0}^{+\infty}  \sqrt{P(\theta^T X \leq -t)P(\theta^T X>-t) }  dt \geq \sqrt{1-\epsilon} \int_{t_0}^{+\infty} \sqrt{P(\theta^T X \leq -t) }  dt. \]
Thus
\begin{align*}
&\int_{t_0}^{+\infty}  \sqrt{P(\theta^T X \leq t)P(\theta^T X>t) }+\sqrt{P(\theta^T X  
	\leq -t)P(\theta^T X>-t) }dt \\
&\geq  \sqrt{1-\epsilon} \int_{t_0}^{+\infty} \sqrt{P(\theta^T X > t) }+ \sqrt{P(\theta^T X \leq -t) }  dt\\
&\geq  \sqrt{1-\epsilon} \int_{t_0}^{+\infty} \sqrt{P(\theta^T X > t) +P(\theta^T X \leq -t) }  dt\\
&\geq  \sqrt{1-\epsilon} \int_{t_0}^{+\infty} \sqrt{P(\theta^T X > t) +P(\theta^T X < -t) }  dt\\
&= \sqrt{1-\epsilon} \int_{t_0}^{+\infty} \sqrt{P(|\theta^T X |> t) }  dt.
\end{align*}
The second inequality is from $\sqrt{a}+\sqrt{b} \geq \sqrt{a+b}$.\\
From above, we have
\begin{align*}
	\int_{-\infty}^{+\infty} \sqrt{F(\theta,t)(1-F(\theta,t))}dt &\geq \sqrt{1-\epsilon} \int_{t_0}^{+\infty} \sqrt{P(|\theta^T X |> t) }dt \\
	&\geq \sqrt{1-\epsilon} \int_{0}^{+\infty} \sqrt{P(|\theta^T X |> t) }dt- t_0\sqrt{1-\epsilon} .
\end{align*}
Now, we can assert that
\[ \int_{S^{d-1}} \int_{-\infty}^{+\infty} \sqrt{F(\theta,t)(1-F(\theta,t))}dtd\sigma(\theta) < \infty \Rightarrow \int_{S^{d-1}} \int_{-\infty}^{+\infty} \sqrt{P(|\theta^T X|>t)}dtd\sigma(\theta) < \infty.\]

\textbf{Step 2:} In this step,  we are going to prove that
\[ \int_{S^{d-1}} \int_{0}^{+\infty}\sqrt{P(|\theta^T X|>t)}dtd\sigma(\theta) <\infty \Leftrightarrow \int_{0}^{+\infty} \sqrt{P(\|X\|>t)}dt <+\infty.\]

"$\Leftarrow$": It follows directly from $|\theta^T X| \leq \|X\|$.

"$\Rightarrow$":

From $\int_{S^{d-1}} \int_{0}^{+\infty}\sqrt{P(|\theta^T X|>t)}dtd\sigma(\theta)<\infty$, we have 
 \[\int_{0}^{+\infty}\sqrt{P(|\theta^T X|>t)}dt <\infty, \ for \ a.e. \ \theta \in S^{d-1}.\]
 
Notice that by the compactness of $S^{d-1}$, we can pick a finite set $\{ \theta_1, \cdots, \theta_n \} \subset S^{d-1}$ such that
$ \int_{0}^{+\infty}\sqrt{P(|{\theta}_i^T X|>t)}dt <\infty \ (1\leq i \leq n)$ and for all $x \in \mathbb{R}^d$, the following inequality holds
\[ \frac{\|x\|}{2} \leq \max_{1 \leq i \leq n}|\theta_i^T x| \leq \sum\limits_{i=1}^{n} |\theta_i^T x|.\]
Thus, we have
\begin{align*}
\sqrt{P(\|X\|>t)}&\leq \sqrt{P(\sum\limits_{i=1}^{n} |\theta_i^T X| >\frac{t}{2})} \\
&\leq \sqrt{\sum\limits_{i=1}^{n} P(|\theta_i^T X|>\frac{t}{2n}) }\\
&\leq \sum\limits_{i=1}^{n} \sqrt{P(|\theta_i^T X|>\frac{t}{2n})}.
\end{align*}
It leads to 
\[ \int_{0}^{+\infty} \sqrt{P(\|X\|>t)} dt \leq  \sum\limits_{i=1}^{n} \int_{0}^{+\infty} \sqrt{P(|\theta_i^T X|>\frac{t}{2n})}dt <\infty.\]
It follows from step 1 and step 2 that 
\[ \int_{S^{d-1}} \int_{-\infty}^{+\infty} \left(E[|Y(\theta,t)|^2]  \right)^{1/2}
dt d\sigma(\theta) < \infty  \Leftrightarrow \int_{0}^{+\infty} \sqrt{P(\|X\|>t)}dt <+\infty.\]

Therefore, we only need to prove that
\[\int_{0}^{+\infty} \sqrt{P(\|X\|>t)}dt <+\infty
\Rightarrow P(\|Y\|_{L^1} >t) = o(t^{-2}).\]
Notice that $\|Y\|_{L^1}\leq \mathbb{E}\|X\|+4\|X\|$, it's sufficient to prove 
\[ \int_{0}^{+\infty} \sqrt{P(\|X\|>t)}dt <+\infty
\Rightarrow \mathbb{E}\|X\|^2 < \infty,\]
which is obvious. Thus we complete the proof.
\end{proof}

\subsection{Proof of Theorem 1(ii)}
\begin{proof}
The proof is just a minor modification of the proof of Theorem 2.1(b) in \cite{1}. 
Let $Z(\theta, t) = I_{\theta^T X>t} $, then $Y(\theta, t)=Z(\theta,t)-\mathbb{E}Z(\theta,t)$.

For necessity, from the proof of Theorem 2.1(b) in \cite{1} (i.e. combining with Hoffmann-Jørgensen's inequality and Montgomery-Smith's inequality, the stochastical boundness implies the boundness of the first moment), we have
\[ \sup_{n} \mathbb{E}\left\| \frac{ \sum_{i=1}^{n}(Z_i-\mathbb{E}Z_i)   }{\sqrt{n}} \right \|_{L^1} <\infty,\]
and for a binomial $(n,p)$ random variable $\xi$ such that $4c_1/n \leq p \leq 1/2$, we have
\[ \frac{np}{2}\leq \mathbb{E}(\xi-\mathbb{E}\xi)^2 \leq c_1[1+(c_2^{-1} \mathbb{E}|\xi-\mathbb{E}\xi|)^2] \leq \frac{np}{4} + \frac{c_1}{c_2^2}(\mathbb{E}|\xi-\mathbb{E}\xi|)^2.\]
In other words, there exist positive finite constants $C_1$ and $C_2$ such that 
\begin{align}
L(\xi)=Bin(n,p) \ with \ \frac{C_1}{n} \leq p \leq \frac{1}{2} \ implies \ E|\xi-E\xi| \geq C_2\sqrt{np}.
\end{align}

For real-valued random variables $\hat{X}, \hat{X_1},\cdots, \hat{X_n}$, 

when $t < 0$, observe that $I_{\hat{X}^{+}>t}-P(\hat{X}^+ > t)=0$, so we have $|\sum_{i=1}^n I_{\hat{X}_i>t}-P(\hat{X}_i>t)| \geq 0=|\sum_{i=1}^n I_{\hat{X_i}^+>t}-P(\hat{X_i}^+>t)|$;

when $t \geq 0$, we have $I_{\hat{X}>t}=I_{\hat{X}^+ > t},\ P(\hat{X}>t)= P(\hat{X}^+>t)$, which implies $|\sum_{i=1}^n I_{\hat{X_i}>t}-P(\hat{X_i}>t)| =|\sum_{i=1}^n I_{\hat{X_i}^+>t}-P(\hat{X_i}^+>t)|$.

Combining above, we have $|\sum_{i=1}^n I_{\hat{X}_i>t}-P(\hat{X}_i>t)| \geq |\sum_{i=1}^n I_{\hat{X}_i^+>t}-P(\hat{X}_i^+>t)|$ holds for any $ t \in \mathbb{R}$. 

Let $Z^+(\theta, t) = I_{(\theta^T X)^+ >t}$. From above, we have $|\sum_{i=1}^n (Z_i^+ -\mathbb{E} Z_i^+) | \leq |\sum_{i=1}^n (Z_i-\mathbb{E}Z_i) |$, which implies that $\|\sum_{i=1}^n \frac{Z_i^+ -\mathbb{E} Z_i^+}{\sqrt{n}}\|_{L^1}$ is also stochastically bounded. Thus (1) yields
\[\sqrt{P((\theta^T X)^+ >t)} \leq \frac{1}{C_2} \mathbb{E} \left| \frac{\sum_{i=1}^n I_{(\theta^T X_i)^+ >t}-P((\theta^T X_i)^+ >t)}{\sqrt{n}}\right|,\]
for $med((\theta^T X)^+) <t<Q_{\theta}(1-\frac{C_1}{n})$, where $Q_{\theta}(t) = \inf\{x: P((\theta^T X)^+ \leq x) \geq t\}$. After integrating both sides, we have
\begin{align*}
&\int_{S^{d-1}}\int_{med((\theta^T X)^+)}^{Q_{\theta}(1-\frac{C_1}{n})} \sqrt{P((\theta^T X)^+ >t)} dtd\sigma(\theta) \\
&\leq \frac{1}{C_2}
\int_{S^{d-1}} \int_{med((\theta^T X)^+)}^{Q_{\theta}(1-\frac{C_1}{n})} \mathbb{E} \left|\frac{\sum_{i=1}^n I_{(\theta^T X_i)^+ >t}-P((\theta^T X_i)^+ >t)}{\sqrt{n}} \right|  dtd\sigma(\theta)\\
&\leq \frac{1}{C_2}  \int_{S^{d-1}} \int_{0}^{+\infty} \mathbb{E}\left|\frac{\sum_{i=1}^n I_{(\theta^T X_i)^+ >t}-P((\theta^T X_i)^+ >t)}{\sqrt{n}}\right|  dtd\sigma(\theta)\\
&= \frac{1}{C_2}  E\left\| \frac{\sum_{i=1}^n (Z_i^+-\mathbb{E}Z_i^+)}{\sqrt{n}} \right\|_{L^1} .
\end{align*}
It implies that 
\[\sup_{n} \int_{S^{d-1}}\int_{med((\theta^T X)^+)}^{Q_{\theta}(1-\frac{C_1}{n})} \sqrt{P((\theta^T X)^+ >t)} dtd\sigma(\theta)
\leq \sup_{n} \frac{1}{C_2}  \mathbb{E}\left\| \frac{\sum_{i=1}^n (Z_i^+-\mathbb{E}Z_i^+)}{\sqrt{n}} \right\|_{L^1} < \infty.\]
Since $Q_{\theta}(1-\frac{C_2}{n}) \rightarrow ess\sup (\theta^T X)^+$ as $n \to \infty$. We have
\begin{align*}
& \int_{S^{d-1}}\int_{med((\theta^T X)^+)}^{+\infty} \sqrt{P((\theta^T X)^+ >t)} dtd\sigma(\theta)\\
&= \int_{S^{d-1}} \liminf_{n \to \infty} \int_{med((\theta^T X)^+)}^{Q_{\theta}(1-\frac{C_1}{n})} \sqrt{P((\theta^T X)^+ >t)} dtd\sigma(\theta) \\
&\leq \liminf_{n \to \infty} \int_{S^{d-1}} \int_{med((\theta^T X)^+)}^{Q_{\theta}(1-\frac{C_1}{n})} \sqrt{P((\theta^T X)^+ >t)} dtd\sigma(\theta) \\
&\leq \sup_{n} \int_{S^{d-1}}\int_{med((\theta^T X)^+)}^{Q_{\theta}(1-\frac{C_1}{n})} \sqrt{P((\theta^T X)^+ >t)} dtd\sigma(\theta)\\
&<\infty,
\end{align*}
where the first inequality follows from the Fatou lemma. In the other hand, from the definition of the median, we have $P((\theta^T X)^+ \geq med((\theta^T X)^+ ) \geq \frac{1}{2}$ and then
\[ P(\|X\|\geq med((\theta^T X)^+ ) \geq P(|\theta^T X|\geq med((\theta^T X)^+) \geq P((\theta^T X)^+ \geq med((\theta^T X)^+ ) \geq \frac{1}{2}.\]
Combining with $\mathbb{E}\|X\|<\infty$ and $P(\|X\| \geq t) \leq \frac{\mathbb{E}\|X\|}{t}$. Let $t= med((\theta^T X)^+ )$, we have $med((\theta^T X)^+ \leq 2\mathbb{E}\|X\|$.

Thus $\int_{S^{d-1}}\int_{2\mathbb{E}\|X\|}^{+\infty} \sqrt{P((\theta^T X)^+ >t)} dtd\sigma(\theta) < \infty$ and it implies that
\[\int_{S^{d-1}}\int_{0}^{+\infty} \sqrt{P((\theta^T X)^+ >t)} dtd\sigma(\theta) < \infty.\]

For $\hat{X}, \hat{X_1}, \cdots, \hat{X_n}$ defined as above.

When $t\leq 0$, we have $I_{\hat{X}^- \geq t}-P(\hat{X}^- \geq t) = 0$, thus $|\sum_{i=1}^{n}I_{\hat{X_i}\leq -t} -P(\hat{X_i}\leq -t)| \geq 0=|\sum_{i=1}^{n} I_{\hat{X_i}^- \geq t}-P(\hat{X_i}^- \geq t)| $.

When $t> 0$, we have $\hat{X} \leq -t \Leftrightarrow \hat{X}^- \geq t$, so $I_{\hat{X}\leq -t} -P(\hat{X}\leq -t) = I_{\hat{X}^- \geq t}-P(\hat{X}^- \geq t)$.

From above, for any $t \in \mathbb{R}$, $|\sum_{i=1}^n I_{\hat{X}_i \leq -t}-P(\hat{X}_i \leq -t)| \geq |\sum_{i=1}^n I_{\hat{X}_i^{-} \geq t}-P(\hat{X}_i^{-} \geq t)| $ holds. 

Let $Z^- = I_{(\theta^T X)^-\geq t }$, for the sequence $\sum_{i=1}^n Y_i$, we have
\begin{align*}
\int_{-\infty}^{+\infty}|\sum_{i=1}^n Y_i| dt&= \int_{-\infty}^{+\infty}|\sum_{i=1}^n (I_{\theta^T X>t} - P(\theta^T X>t))|dt\\
& = \int_{-\infty}^{+\infty}|\sum_{i=1}^n (I_{\theta^T X\leq t} - P(\theta^T X\leq t))|dt\\
&= \int_{-\infty}^{+\infty}|\sum_{i=1}^n (I_{\theta^T X\leq -t} - P(\theta^T X\leq -t))|dt\\
&\geq \int_{-\infty}^{+\infty}|\sum_{i=1}^n(I_{(\theta^T X)^-\geq t} - P((\theta^T X)^-\geq t))|dt,
\end{align*}
which implies that the sequence $\|\sum_{i=1}^n \frac{Z_i^- - \mathbb{E}Z_i^-}{\sqrt{n}}\|_{L^1}$ is also stochastically bounded. Thus similarly, we can obtain that
\[ \int_{S^{d-1}}\int_{0}^{+\infty} \sqrt{P((\theta^T X)^- >t)} dtd\sigma(\theta)<\infty.\]
Since
\begin{align*}
	\sqrt{P(|\theta^T X| >t)} &= \sqrt{P((\theta^T X)^+ + (\theta^T X)^- >t)} \\
	&\leq \sqrt{P((\theta^T X)^+ > \frac{t}{2}) +P((\theta^T X)^- > \frac{t}{2})}	\\
	&\leq \sqrt{P((\theta^T X)^+ > \frac{t}{2}) }+ \sqrt{P((\theta^T X)^- > \frac{t}{2}) }  .
\end{align*}
Now we can conclude that $\int_{S^{d-1}}\int_{0}^{+\infty} \sqrt{P(|\theta^T X|>t)} dtd\sigma(\theta)<\infty$. And it leads to that $\int_{0}^{+\infty} \sqrt{P(\|X\| >t)}dt < \infty$.

Sufficiency: If $\Lambda_{2,1}(X)<\infty$, Corollary 1 implies that there exists a constant $M>0$ such that $\sup_{n} \mathbb{E}[\sqrt{n}SW_1(\mu_n,\mu)]\leq M$. Therefore, the conclusion follows from Markov inequality. \\
\end{proof}

\subsection{Proof of Proposition 1}

\begin{proof}
(1): The proof is based on the method in Proposition 2.1 in \cite{60} with just a little change. \cite{61} explored the central limit theorem for $W_1$ in one dimension under $i.i.d.$ and non-$i.i.d.$ conditions. Due to the similarity between $W_1$ in one dimension and $SW_1$, the method also works for $SW_1$. 

Following from the same decomposition:
\[ |x+h|-|x|=hI_{x+h\geq 0, x>0}- hI_{x+h<0, x<0}+|h|I_{x=0}+(|x+h|-|x|)(I_{x+h\geq 0, x<0}+I_{x+h<0, x>0}), \]
which yields that
\begin{align*}
|x+h|-|x|=&hI_{x>0}-hI_{x<0}+|h|I_{x=0}-hI_{x+h<0,x>0}+hI_{x+h\geq 0, x<0} \\
& (|x+h|-|x|)(I_{x+h\geq 0, x<0}+I_{x+h<0,x>0}).
\end{align*}

Then we have
\begin{equation}
|x+h|-|x|=hI_{x>0}-hI_{x<0}+|h|I_{x=0}+2R(h,x),
\end{equation}
where $R(h,x)\leq |h|(I_{x+h\geq 0, x<0}+I_{x+h<0, x>0})$.

Let $x=F(\theta, t)-G(\theta,t)$, $h=F_n(\theta,t)-F(\theta,t)$ in (2), we can deduce that 
\[\sqrt{n}(|F_n-G|-|F-G|)=sign\{F-G\}\sqrt{n}(F_n-F)I_{F\neq G} +\sqrt{n}|F_n-F|I_{F=G}+2R_n,\]
where 
\begin{equation}
|R_n|\leq \sqrt{n}|F_n-F|(I_{F_n\geq G, F< G}+I_{F_n<G, F>G}),
\end{equation}
and
\begin{equation*}
sign(x)=\left\{\
\begin{aligned}
1&, &x>0, \\
0&, &x=0, \\
-1&, &x<0. 
\end{aligned}
\right.
\end{equation*}
Recall that
\[\sqrt{n}(SW_1(\mu_n, \nu)-SW_1(\mu, \nu))=\sqrt{n}\int_{S^{d-1}}\int_{-\infty}^{\infty} (|F_n(\theta,t)-G(\theta,t)|-|F(\theta,t)-G(\theta,t)|)ddtd\sigma(\theta). \]

By central limit theorem in $L^1(S^{d-1}\times \mathbb{R})$ and continuous mapping theorem (notice that $sign\{F-G\}, I_{F\neq G}, I_{F=G} \in L^{\infty}(S^{d-1}\times \mathbb{R})$), it's sufficient to prove that $\|R_n\|_{L^1}$ converges to 0 in probability as $n \to \infty$. 

On the other hand, provided $\mu \in \mathcal{P}_{2,1}(\mathbb{R}^d)$, by Fubini theorem and Cauchy inequality, we have
\begin{align*}
\mathbb{E}[\|R_n\|_{L^1}] &=\int_{S^{d-1}}\int_{-\infty}^{\infty} \mathbb{E}|R_n(\theta,t)|dtd\sigma(\theta) \\
&\leq \int_{S^{d-1}}\int_{-\infty}^{\infty} (\mathbb{E}|R_n(\theta,t)|^2)^{1/2} dtd\sigma(\theta)\\
&\leq \int_{S^{d-1}}\int_{-\infty}^{\infty} \sqrt{P(|\theta^T X_1|>t)}dtd\sigma(\theta)<\infty.
\end{align*}
With dominated convergence theorem, it's sufficient to prove that $\mathbb{E}|R_n(\theta,t)|$ converges to 0 as $n \to \infty$ for any $(\theta,t)\in S^{d-1}\times \mathbb{R}$. 

Let $T_n(\theta,t)=\sqrt{n}|F_n-F|(I_{F_n\geq G, F< G}+I_{F_n<G, F>G})$. By (3) and Markov inequality, we have 
\begin{align*}
\mathbb{E}|R_n(\theta,t)| \leq \mathbb{E} T_n(\theta,t)I_{T_n(\theta,t)>M}&+MP(F_n(\theta,t)\geq G(\theta,t),F(\theta,t)<G(\theta,t) )\\
&+MP(F_n(\theta,t)<G(\theta,t), F(\theta,t)>G(\theta,t)).
\end{align*}
Since $F_n(\theta,t)$ converges almost surely to $F(\theta,t)$, we can assert that
\[\lim\limits_{n\to \infty} P(F_n(\theta,t)\geq G(\theta,t),F(\theta,t)<G(\theta,t) )+P(F_n(\theta,t)<G(\theta,t), F(\theta,t)>G(\theta,t))=0. \]
Moreover, by Markov inequality, we have
\[\mathbb{E} T_n(\theta,t)I_{T_n(\theta,t)>M} \leq \frac{\mathbb{E}T_n(\theta,t)^2}{M} \leq \frac{1}{M}. \]
Combining above, we have
\[\lim\limits_{n\to \infty}\lim\sup_{n \to \infty}\mathbb{E}T_n(\theta,t)I_{T_n(\theta,t)>M}=0 .\]
Therefore, $\mathbb{E}|R_n(\theta,t)|$ converges to 0 as $n \to \infty$ for any $(\theta,t)\in S^{d-1}\times \mathbb{R}$, which yields the conclusion.

(2) Let $x=F(\theta,t)-G(\theta,t), h=(F_n(\theta,t)-G_n(\theta,t))-(F(\theta,t)-G(\theta,t))$, by the same proof in (1), we have the conclusion. 
\end{proof}

\subsection{Proof of Proposition 2}
\begin{proof}
First, Fubini theorem implies
\begin{align*}
\mathbb{E}[SW_1(\mu_n, \mu)] &= \int_{S^{d-1}}\mathbb{E}[W_1(\theta_{\#}^{*}\mu_n,\theta_{\#}^{*}\mu )] d\sigma(\theta) \\
&= \int_{S^{d-1}} \mathbb{E}[\int_{-\infty}^{+\infty} |F_n(\theta, t)-F(\theta, t)|dt]d\sigma(\theta).
\end{align*}
To invoke dominated convergence theorem, we need to estimate the part inside the integral. 
\begin{align*}
&\mathbb{E}\left[\int_{-\infty}^{+\infty} |F_n(\theta, t)-F(\theta, t)|dt\right]\\
&= \int_{-\infty}^{+\infty} \mathbb{E}[  |F_n(\theta, t)-F(\theta, t)|       ]dt \\
&= \int_{-\infty}^{0}\mathbb{E}[  |F_n(\theta, t)-F(\theta, t)|       ]dt   + \int_{0}^{+\infty}\mathbb{E}[  |F_n(\theta, t)-F(\theta, t)|       ]dt .
\end{align*}
For $t \in(-\infty, 0]$, we have
\begin{align*}
\mathbb{E}[  |F_n(\theta, t)-F(\theta, t)|       ] &\leq \mathbb{E}[F_n(\theta,t)+ F(\theta,t)] \\
&= 2F(\theta,t) \\
& = 2P(\theta^T X \leq t) \leq 2P(\|X\|\geq |t|) .
\end{align*}
For $t \in (0,+\infty)$, we have
\begin{align*}
\mathbb{E}[  |F_n(\theta, t)-F(\theta, t)|       ] &\leq \mathbb{E}[1-F_n(\theta,t)+ 1-F(\theta,t)] \\
&= 2(1-F(\theta,t)) \\
& = 2P(\theta^T X >t) \leq 2P(\|X\|>t) .
\end{align*}
Combining above, we obtain that
\[\mathbb{E}\left[\int_{-\infty}^{+\infty} |F_n(\theta, t)-F(\theta, t)|dt\right] \leq 4 \int_{0}^{+\infty}P(\|X\| \geq t) dt = 4\mathbb{E}[\|X\|].\]

Theorem 2.14 in \cite{13} gives the same result for the Wasserstein distance in one dimension, which implies that $\mathbb{E}[W_1(\theta_{\#}^{*}\mu_n,\theta_{\#}^{*}\mu )] \to 0, as \ n \to \infty$ for any $\theta \in S^{d-1}$. Thus the conclusion follows from the dominated convergence theorem. 
\end{proof}

\subsection{Proof of Corollary 1}
\begin{proof}
Sufficiency: 
\begin{align*}
\mathbb{E}[SW_1(\mu_n, \mu)] &=\int_{S^{d-1}}\int_{-\infty}^{+\infty} \mathbb{E}|F_n(\theta,t)-F(\theta, t)|dtd\sigma(\theta) \\
&\leq \int_{S^{d-1}}\int_{-\infty}^{+\infty} \left( \mathbb{E}|F_n(\theta,t)-F(\theta, t)|^2\right)^{1/2}dtd\sigma(\theta) \\
&=\frac{1}{\sqrt{n}} \int_{S^{d-1}}\int_{-\infty}^{+\infty} \sqrt{F(\theta,t)(1-F(\theta,t))}dtd\sigma(\theta).
\end{align*}
Thus $\mathbb{E}[SW_1(\mu_n, \mu)]=\mathcal{O}(\frac{1}{\sqrt{n}})$, 
Since 
\[\Lambda_{2,1}(X)<\infty \Leftrightarrow \int_{S^{d-1}}\int_{-\infty}^{+\infty} \sqrt{F(\theta,t)(1-F(\theta,t))}dtd\sigma(\theta)<\infty.\]
Necessity: $\mathbb{E}[SW_1(\mu_n),\mu]=\mathcal{O}(\frac{1}{\sqrt{n}})$ implies that the sequence in Theorem 1(ii) is stochastically bounded, so $\Lambda_{2,1}(X)<\infty$. 
\end{proof}

\subsection{Proof of Proposition 3}
By Lemma 3.4 in \cite{13}, that is for any independent mean zero random variables $\xi_1, \cdots, \xi_n$, the following inequality holds.
\[\mathbb{E}|\sum\limits_{k=1}^n\xi_k| \geq c \mathbb{E} \left(\sum\limits_{k=1}^n\xi_k^2 \right)^{1/2}, \]
where $c$ is an absolute constant.

Note that 
\[\mathbb{E}[SW_1(\mu_n,\mu)] =\int_{S^{d-1}} \int_{-\infty}^{+\infty} \mathbb{E}|F_n(\theta,t)-F(\theta, t)|dtd\sigma(\theta).\]
And let $\xi_k=I_{\theta^T X_i\leq t}-F(\theta,t)$, we have
\begin{align*}
\mathbb{E}|F_n(\theta,t)-F(\theta, t)| &\geq \frac{c}{n}\mathbb{E} \left(\sum\limits_{k=1}^n\xi_k^2 \right)^{1/2}\\
&\geq\frac{c}{n} \left( \sum\limits_{k=1}^n (\mathbb{E}|\xi_k|)^2 \right)^{1/2}\\
&= \frac{2c}{\sqrt{n}} F(\theta,t)(1-F(\theta,t)),
\end{align*}
where the second inequality follows from the Minkowski inequality.

Then 
\[ \mathbb{E}[SW_1(\mu_n,\mu)] \geq \frac{2c}{\sqrt{n}} \int_{S^{d-1}} \left[\int_{-\infty}^{+\infty} F(\theta,t)(1-F(\theta,t))dt\right]d\sigma(\theta).\]

Note that for a real-valued random variable $Y$ with CDF $G(y)$ and let $Y^{'}$ be an independent copy of $Y$, the equality $\int_{-\infty}^{\infty}G(y)(1-G(y))dy=\frac{1}{2}\mathbb{E}|Y-Y^{'}|$ holds. Let $X\sim \mu$ and $X^{'}$ be an independent copy of $X$, then we have 
\begin{align*}
\mathbb{E}[SW_1(\mu_n,\mu)] &\geq \frac{c}{\sqrt{n}} \int_{S^{d-1}} \mathbb{E}|\theta^T X- \theta^T X^{'}|d\sigma(\theta) \\
&= \frac{c}{\sqrt{n}}\mathbb{E}  \left[\int_{S^{d-1}}|\theta^T X- \theta^T X^{'}|d\sigma(\theta) \right]\\
&=\frac{c\gamma_d}{\sqrt{n}} \mathbb{E}\|X-X^{'}\|       ,
\end{align*}
where $\gamma_d$ is a constant depending only on $d$.

In the next, we write $\mathbb{E}_{X}$ for taking expectation for $X$, then we have
\begin{align*}
\mathbb{E}\|X-\mathbb{E}X\|&=\mathbb{E}_{X}\|X-\mathbb{E}_{X}X-\mathbb{E}_{X^{'}}(X^{'}-\mathbb{E}_{X^{'}}X^{'}) \|\\
&=\mathbb{E}_{X}\| \mathbb{E}_{X^{'}}(X-\mathbb{E}_{X}X)-(X^{'}-\mathbb{E}_{X^{'}}X^{'}) \|\\
&\leq \mathbb{E}_{X,X^{'}} \| (X-\mathbb{E}_{X}X)-(X^{'}-\mathbb{E}_{X^{'}}X^{'}) \|\\
&= \mathbb{E} \|X-X^{'}\|.
\end{align*}

Combining above, we conclude.

\subsection{Proof of Proposition 4} 
\begin{proof}
Let $F(\theta) := W_p^p(\theta_{\#}^{*} \mu, \theta_{\#}^{*} \nu)$, then $SW_p^p(\mu, \nu) = \mathbb{E}_{\sigma} F(\theta)$. We firstly prove that $F(\theta)$ is a Lipschitz continuous function with Lipschitz constant at most $pW_p^{p-1}(\mu, \nu)(M_p(\mu)+M_p(\nu))$.
\begin{align*}
|F(\theta_1)-F(\theta_2) | &= |W_p^p({\theta_1^*}_{\#}\mu, {\theta_1^*}_{\#}\nu ) -W_p^p({\theta_2^*}_{\#}\mu, {\theta_2^*}_{\#}\nu ) | \\
&=p\xi^{p-1} |W_p({\theta_1^*}_{\#}\mu, {\theta_1^*}_{\#}\nu ) -W_p({\theta_2^*}_{\#}\mu, {\theta_2^*}_{\#}\nu ) | \\
&\leq p\xi^{p-1} [W_p({\theta_1^*}_{\#}\mu, {\theta_2^*}_{\#}\mu )+W_p({\theta_1^*}_{\#}\nu, {\theta_2^*}_{\#}\nu )]\\
&\leq pW_p^{p-1}(\mu, \nu)[W_p({\theta_1^*}_{\#}\mu, {\theta_2^*}_{\#}\mu )+W_p({\theta_1^*}_{\#}\nu, {\theta_2^*}_{\#}\nu )]
\end{align*}
where the second equality follows from the Lagrange Mean value Theorem and the last inequality follows by that for $W_p({\theta^*}_{\#}\mu, {\theta^*}_{\#}\nu ) \leq W_p(\mu, \nu)$ for any $\theta \in S^{d-1}$.

Notice that $({\theta_1^*}, {\theta_2^*})_{\#} \mu \in \prod({\theta_1^*}_{\#}\mu, {\theta_2^*}_{\#}\mu) $ where $({\theta_1^*}, {\theta_2^*})(x)=({\theta_1^*}(x), {\theta_2^*}(x))$. Thus we have
\begin{align*}
W_p({\theta_1^*}_{\#}\mu, {\theta_2^*}_{\#}\mu) & \leq \left(\int_{\mathbb{R}\times \mathbb{R}} |x-y|^p d({\theta_1^*}, {\theta_2^*})_{\#} \mu(x,y) \right)^{\frac{1}{p}} \\
&= \left( \int_{\mathbb{R}^d} |\theta_1^*(z)-\theta_2^*(z) |^p d\mu(z)   \right)^{\frac{1}{p}} \\
&= \left( \int_{\mathbb{R}^d} |\langle \theta_1-\theta_2, z \rangle|^p d\mu(z)   \right)^{\frac{1}{p}} \\
& \leq \|\theta_1-\theta_2\| \left( \int_{\mathbb{R}^d} \|z\|^p d\mu(z) \right)^{\frac{1}{p}} \\
& = \|\theta_1-\theta_2\| M_p(\mu).
\end{align*}
Similarly, we have $W_p({\theta_1^*}_{\#}\nu, {\theta_2^*}_{\#}\nu)  \leq |\theta_1-\theta_2| M_p(\nu)$. Thus, we obtain
\begin{align*}
|F(\theta_1)-F(\theta_2)| &\leq pW_p^{p-1}(\mu, \nu)(M_p(\mu)+M_p(\nu)) \|\theta_1-\theta_2\|  \\
&\leq pW_p^{p-1}(\mu, \nu)(M_p(\mu)+M_p(\nu)) \rho(\theta_1, \theta_2)
\end{align*}
where $\rho(x, y):= \arccos \langle x,y \rangle$ is the geodesic distance in $S^{d-1}$.

In the following, we need the concentration inequality of uniform distribution over the $d$-dimensional unit sphere. By seeing $S^{d-1}$ as a compact connected smooth Riemannian manifold, then Proposition 2.17 in \cite{59} tells us that  for any 1-Lipschitz continuous function $f$ on $S^{d-1}$ and $X\sim \sigma$, $f(X)$ is sub-Gaussian with parameter ${\widetilde{\sigma}}^2=\frac{1}{c(S^{d-1})}=\frac{1}{d-1}$, where $c(S^{d-1})$ is infimum of Ricci curvature tensor over all tangent vectors. Then
from Hoeffding's inequality (Proposition 2.5 in \cite{14}), we have 
\[ P \left(  | \frac{F(\theta_1)+\cdots+ F(\theta_n)}{n} -\mathbb{E}_{\sigma}[F(\theta)]   | \geq   \epsilon \right)
\leq  2 e^{-\frac{(d-1)n \epsilon^2}{2L^2}} .\]

Therefore if requiring  $P \left(  | \frac{F(\theta_1)+\cdots+ F(\theta_n)}{n} -\mathbb{E}_{\sigma}[F(\theta)]   | \geq   \epsilon \right) \leq \delta $, we can let 
$2 e^{-\frac{(d-1)n \epsilon^2}{2L^2}} \leq \delta$, i.e. $ n \geq \frac{2L^2}{(d-1) \epsilon^2 } \log\frac{2}{\delta}$. 
\end{proof}

\subsection{Proof of Corollary 3}
\begin{proof}
\begin{align*}
M_2^2(\mu) &= \int_{\mathbb{R}^d} x_1^2 + \cdots + x_d^2 d \mu(x) \\
&= \int_{\mathbb{R}^d} f\circ \pi_1(x)  +\cdots + f\circ \pi_d(x) \mu(x) \\
&= \int_{\mathbb{R}} f(y)d(\pi_i)_{\#} \mu(y)+\cdots + f(y)d(\pi_d)_{\#} \mu(y) \\
&= \sum\limits_{i=1}^d M_2^2((\pi_i)_{\#}\mu) \leq d \max_i M_2^2((\pi_i)_{\#}\mu).
\end{align*}
Combining with Corollary 3 and $\frac{d}{d-1}\leq 2$ for $d\geq 2$, we have the conclusion. \end{proof}

\subsection{Proof of Corollary 5}
\begin{proof}
Let 
\[ \overline{SW}_{p}^p(\mu, \nu) := \int_{\mathbb{R}^d} W_p^p({\theta}_{\#}^{*} \mu, {\theta}_{\#}^{*} \nu) d \beta_d(\theta).\]
where $\beta_d$ is the zero-mean Gaussian distribution with covariance matrix $I_d$.

Then by calculation similar to Proposition 1 in \cite{6}, we have $\overline{SW}_{p}^p(\mu, \nu) = d^{p/2} \widetilde{SW}_p^p(\mu,\nu)$. Thus the conclusion follows by the concentration properties of Lipschitz functions of Gaussian variables (Theorem 2.26 in \cite{14}).
\end{proof}

\subsection{Proof of Lemma 1}
\begin{proof}
\begin{align*}
MSW_1(P_n, P) &= \max_{\theta \in S^{d-1} } W_1( {\theta}_{\#}^{*} P_n, {\theta}_{\#}^{*} P) \\
&= \max_{\theta \in S^{d-1} } \sup_{g \in Lip_{1,0}(\mathbb{R})} \int g d {\theta}_{\#}^{*} P_n - \int g d {\theta}_{\#}^{*} P \\
&= \max_{\theta \in S^{d-1} } \sup_{g \in Lip_{1,0}(\mathbb{R})} \int g(\theta^T x) d P_n(x) - \int g(\theta^T x) d P(x) \\
&= \sup_{f \in \mathcal{F}} \int f d P_n - \int f d P \\
&= \sup_{ f \in \mathcal{F}} (P_n-P)f,
\end{align*}
where 
\[\mathcal{F}:= \{f :\mathbb{R}^d\to \mathbb{R}| \ f(x)=g(\theta^T x), g \in Lip_{1,0}(\mathbb{R}), \theta \in S^{d-1} \}.\]
\end{proof}
Before bounding the bracket number of $\mathcal{F}$, we need the following lemmas: 

\begin{lemma}
	For any positive constant $M$, we have 
\[\mathcal{N}_{[ \ ]}(\epsilon, Lip_{1,0}([-M, M]), \| \cdot \|_{\infty}) \leq e^{ \frac{4M}{\epsilon}\log2 }.\]
\end{lemma}
\begin{lemma} 
$\mathcal{N}(\epsilon, S^{d-1}, \|\cdot\|) \leq  (1+\frac{4}{\epsilon})^d$.
\end{lemma}

\begin{proof}[\textbf{Proof of Lemma 2:}]
First, we divide $[0, M]\times [-M, M]$ into a grid of squares with side length $\frac{\epsilon}{2}$. In $[0,\frac{\epsilon}{2} ]$, let $g(x) = -x, f(x)=x$, then for any $h \in Lip_{1,0}([0, M])$ and $x \in [0,\frac{\epsilon}{2} ]$, $g(x) \leq h(x) \leq f(x)$ holds, which means in $\in [0,\frac{\epsilon}{2} ]$, $ h \in [g, f]$.

In $[\frac{\epsilon}{2}, 2\frac{\epsilon}{2}]$, let $f_1(x) = \epsilon-x$ and $g_1(x) = -\epsilon+x$. For any $h \in Lip_{1,0}([0, M])$, we are going to prove that in $[\frac{\epsilon}{2}, 2\frac{\epsilon}{2}]$ there don't exist two points in $h$'s graph
such that one is above strictly $f_1$'s graph and the other is below strictly $g_1$'s graph. If not, there exist $x_0$ and $x_1$ such that $h(x_0) > f_1(x_0)$ and $h(x_1) < g_1(x_1)$. Without loss of generality, we can assume that $x_0 < x_1$. The Lipschitz continuity of $h$ implies that $h(x_1)-h(x_0) \geq x_0-x_1$, thus we have
$$h(x_1) \geq h(x_0)+x_0-x_1 >f_1(x_0)+x_0-x_1 =\epsilon-x_1 =f_1(x_1) \geq g_1(x_1),$$
which contradicts the assumption $h(x_1) < g_1(x_1)$. 

So there are only three cases:
\begin{itemize}
\item There exist points in $h$'s graph that are above strictly $f_1$'s graph. 

By the consideration above, we can assert that $h(x) \geq g_1(x)$ for all $x \in [\frac{\epsilon}{2}, 2\frac{\epsilon}{2}]$. It implies that in $[\frac{\epsilon}{2}, 2\frac{\epsilon}{2}]$, $h \in [g_1, f_2]$, where $f_2(x) := x$. 

\item There exist points in $h$'s graph that are below strictly $g_1$'s graph.

Thus $h(x) \leq f_1(x)$. As a consequence, $h \in [g_2, f_1]$, where $g_2(x) :=-x$. 

\item There exist points in $h$'s graph that are below $f_1$'s graph and above $g_1$'s graph. 

Then $h$ satisfies that 
$h \in [g_1, f_1]$.
\end{itemize}

From the construction above, we can assert that if in $[0, \frac{\epsilon}{2}]$, $h \in [g, f]$, then 
in $[\frac{\epsilon}{2}, 2\frac{\epsilon}{2}]$, $h \in [-\frac{\epsilon}{2}+g(\frac{\epsilon}{2})+x, \ -\frac{\epsilon}{2}+f(\frac{\epsilon}{2})+x] $ or $h \in [\frac{\epsilon}{2}+g(\frac{\epsilon}{2})-x, \ \frac{\epsilon}{2}+f(\frac{\epsilon}{2})-x]$. Similarly, we can conclude that if in $[0, k\frac{\epsilon}{2}]$, $h \in [g, f]$, then in  $[k\frac{\epsilon}{2}, (k+1)\frac{\epsilon}{2}]$, $h \in  [-k \frac{\epsilon}{2}+g(k\frac{\epsilon}{2})+x,\ -k \frac{\epsilon}{2}+f(k\frac{\epsilon}{2})+x] $ or $h \in [k \frac{\epsilon}{2}+g(k\frac{\epsilon}{2})-x,\ k \frac{\epsilon}{2}+f(k\frac{\epsilon}{2})-x]$. The conclusion implies that if we have constructed the brackets in $[0, k\frac{\epsilon}{2}]$, then we can construct the brackets in
$[0, (k+1)\frac{\epsilon}{2}]$ which are only doubled.

Thus by above, we have
\[\mathcal{N}_{[ \ ]}(\epsilon, Lip_{1,0}([0, M]), \| \cdot \|_{\infty}) \leq 2^{\lceil  \frac{M}{\frac{\epsilon}{2}} \rceil -1} \leq 2^{\frac{2M}{\epsilon}} = e^{\frac{2M}{\epsilon}\log2}. \]
It's easy to see that 
\[\mathcal{N}_{[ \ ]}(\epsilon, Lip_{1,0}([-M, M]), \| \cdot \|_{\infty}) \leq {\mathcal{N}_{[ \ ]}(\epsilon, Lip_{1,0}([0, M]), \| \cdot \|_{\infty})}^2 \leq e^{\frac{4M}{\epsilon}\log2}. \]
\end{proof}

Lemma 2 implies that $\log \mathcal{N}_{[ \ ]}(\epsilon, Lip_{1,0}([0, M]), \| \cdot \|_{\infty}) = M O(\frac{1}{\epsilon})$. On the other hand, we can obtain a lower bound for the bracketing entropy from the construction above. 

\begin{corollary} 
$\log \mathcal{N}_{[ \ ]}(\epsilon, Lip_{1,0}([-M, M]), \| \cdot \|_{\infty}) = M \Theta(\frac{1}{\epsilon})$.
\end{corollary}
\begin{proof}
It's sufficient to prove the lower bound. In the construction above, we denote the larger function in the $i$-th bracket by $f_i$. Thus we have that $\|f_i-f_j\|_{\infty} = \epsilon, (i\neq j)$ if $M=k \frac{\epsilon}{2}, k \in \mathbb{N}$, which implies
\begin{align*}
\log \mathcal{N}_{[ \ ]}(\frac{\epsilon}{2}, Lip_{1,0}([0, M]) &\geq
\log \mathcal{N}_{[ \ ]}(\frac{\epsilon}{2}, Lip_{1,0}([0, \lfloor \frac{2M}{\epsilon} \rfloor\frac{\epsilon}{2} ]) \\
&\geq \log 2^{\lfloor \frac{2M}{\epsilon} \rfloor -1} \\
&= (\lfloor \frac{2M}{\epsilon} \rfloor-1) \log2 \\
&= M \Omega(\frac{1}{\epsilon}).
\end{align*}
So we have $\log \mathcal{N}_{[ \ ]}(\epsilon, Lip_{1,0}([0, M]), \| \cdot \|_{\infty}) = M \Theta(\frac{1}{\epsilon})$ and obviously
\[\log \mathcal{N}_{[ \ ]}(\epsilon, Lip_{1,0}([-M, M]), \| \cdot \|_{\infty}) = M \Theta(\frac{1}{\epsilon}).\]
\end{proof}

\begin{proof}[\textbf{Proof of Lemma 3:}] 
Notice that for a metric space $(\mathcal{X}, d)$, if $A \subset B \subset \mathcal{X}$, then $\mathcal{N}(\epsilon, A, d) \leq \mathcal{N}(\frac{\epsilon}{2}, B, d)$. Since for the unit sphere $B_2(1)$ in $d$-dimensional Euclidean space, we have $\mathcal{N}(\epsilon, B_2(1), \| \cdot \|) \leq (1+\frac{2}{\epsilon})^d $, thus  
\[\mathcal{N}(\epsilon, S^{d-1}, \|\cdot\|) \leq \mathcal{N}(\frac{\epsilon}{2}, B_2(1), \|\cdot\|) \leq (1+\frac{4}{\epsilon})^d.\]
\end{proof}

\subsection{Proof of Theorem 2}
\begin{proof}
We write $\mathcal{F}_1$ for the set consisting of the $\epsilon$-brackets of $Lip_{1,0}([-M, M])$ constructed in the proof of Lemma 3, where $M$ is a positive constant to be determined. Before constructing brackets for $\mathcal{F}$, we need to extend the functions in brackets of $\mathcal{F}_1$ first. For bracket $[g, f] \in \mathcal{F}_1$, when $x >M$, let $f(x)=-M+f(M)+x, g(x)=M+g(M)-x$ and when $x<-M$, let $f(x)=-M+f(-M)-x,g(x)=M+g(-M)+x$. Next, we construct the brackets of $\mathcal{F}$. Recall that 
\[\mathcal{F}= \{f :\mathbb{R}^d\to \mathbb{R} \ | \ f(x)=g(\theta^T x), g \in Lip_{1,0}(\mathbb{R}), \theta \in S^{d-1} \}.\]
Let $\mathcal{G}$ be a minimal $\epsilon$-cover of $S^{d-1}$ in the sense of Euclidean norm, that is $|\mathcal{G}| = \mathcal{N}(\epsilon, S^{d-1}, \| \cdot \|)$. Let $\mathcal{H} = \{ [ g(\theta^T x) -\epsilon\|  x \|, f(\theta^T x) + \epsilon \|  x \| ] : [g, f] \in \mathcal{F}_1, \ \theta \in \mathcal{G} \}$. In the next, we are going to prove that $\mathcal{H}$ is a $C\epsilon$-bracket cover of $\mathcal{F}$, where $C$ is a constant depending only on the moment of $P$.

Given any $\hat{h} \in \mathcal{F}$, assume that $\hat{h}(\cdot) = h(\theta^T \cdot)$, where $h\in Lip_{1,0}(\mathbb{R}), \ \theta \in S^{d-1}$. Thus by construction of $\mathcal{F}_1$, there exist $[g, f] \in \mathcal{F}_1$ and $\theta_1 \in \mathcal{G}$ such that $g(y) \leq h(y) \leq f(y), \ \forall y \in \mathbb{R}$ and $ \| \theta-\theta_1 \|_2 \leq \epsilon$. We just need to restrict $h$ to $[-M, M]$ and the existence of the bracket follows from the construction and the extension. Then we have 
\begin{align*}
	\hat{h}(x) &= h(\theta^T x) \\
	&= h(\theta^T x)-h(\theta_1^T x) + h(\theta_1^T x) \\
	& \leq |\theta^T x-\theta_1^T x| + h(\theta_1^T x) \\
	& \leq \| \theta-\theta_1\| \|x\|+h(\theta_1^T x) \\
	&\leq \epsilon \|x\|+ f(\theta_1^T x) .
\end{align*} 
Similarly, $\hat{h}(x) \geq -\epsilon \|x\|+g(\theta_1^T x)$.

Next, we start to bound the $L_2(P)$-size of the brackets in $\mathcal{H}$. Define  
\[ c:= \|(f(\theta^T x)+\epsilon\|x\|) - (g(\theta^T x)-\epsilon\|x\|) \|_{L^2(P)}.\] 
Then 
\begin{align*}
	c^2 &= \int \left(f(\theta^T x)- g(\theta^T x)+2 \epsilon\|x\| \right)^2 dP(x) \\
	& \leq 2 \int \left(f(\theta^T x)- g(\theta^T x) \right)^2 dP(x) + 8\epsilon^2\int \|x\|^2 dP(x) .
\end{align*}
The inequality follows by $(a+b)^2 \leq 2(a^2+b^2)$. 

Denote the first term by $I:= \int \left(f(\theta^T x)- g(\theta^T x) \right)^2 dP(x)$. 

For estimating $I$, we divide $I$ into two terms.
\begin{align*}
	I = \int_{ \{ |\theta^T x| > M \} } \left(f(\theta^T x)- g(\theta^T x) \right)^2 dP(x) + \int_{ \{ |\theta^T x| \leq M \} } \left(f(\theta^T x)- g(\theta^T x) \right)^2 dP(x) .
\end{align*}
We write $I_1$ for the first term and $I_2$ for the last term.

From the construction of $f$ and $g$, we have
\[ I_2 \leq P(|\theta^T x| \leq M) \|f-g\|_{\infty}^2 |_{[-M,M]} \leq \epsilon^2.\]

For $I_1$, we have
\begin{align*}
	I_1 &= \int_{ \{ |\theta^T x| > M \} } \left(f(\theta^T x)- g(\theta^T x) \right)^2 dP(x) \\
	& \leq \int_{ \{ |\theta^T x| > M \} } 4\|x\|^2 dP(x) \\
	& \leq \int_{ \{ \|x\| > M \} } 4\|x\|^2 dP(x).
\end{align*}
Since when $y > M$, $-y \leq g(y) \leq f(y) \leq y$ , when $y < -M$, $y \leq g(y) \leq f(y) \leq -y$, then when $|y| > M$, $|f(y)-g(y)| \leq 2|y|$. Thus when $|\theta^T x| > M$, we have
$|f(\theta^T x)- g(\theta^T x) | \leq 2 | \theta^T x| \leq 2\|x\|$, the first inequality holds. The second inequality follows by $\{ x:|\theta^T x| > M \} \subset \{ x: \|x\| > M\}$. 

If the ($2+\delta$)-th moment of $P$ is finite, we have 
\begin{align*}
	I_1 &\leq 4\int_{ \{ \|x\| > M \} } \|x\|^2 dP(x) \\
	& \leq 4 \int_{ \{ \|x\| > M \} } \|x\|^{2+\delta} \frac{1}{M^\delta} dP(x) \\
	& \leq 4 \frac{1}{M^\delta}  \int \|x\|^{2+\delta} dP(x) .
\end{align*}

Let $m_p = \int \|x\|^p dP(x)$, then
\[ I = I_1 + I_2 \leq  4 \frac{1}{M^\delta} m_{2+\delta}+\epsilon^2.\]

Then for $c$, we have
\[ c^2 \leq 2I + 8 \epsilon^2 m_2 \leq 2\epsilon^2 + 8\frac{1}{M^\delta} m_{2+\delta} + 8 \epsilon^2 m_2 .\]

Let $M = (\frac{1}{\epsilon})^{\frac{2}{\delta}}$, i.e. $M^{\delta} = \frac{1}{\epsilon^2}$. Now, we have
\[c^2 \leq 2\epsilon^2 +8 \epsilon^2 m_2+8 \epsilon^2 m_{2+\delta}  = C^2 \epsilon^2,\]
where $C:= \sqrt{2+8m_2+8m_{2+\delta}}$.

With these preparations above, we can estimate $\mathcal{N}_{ [ \ ]} (C\epsilon, \mathcal{F}, L_2(P))$. 
\begin{align*}
	\mathcal{N}_{ [ \ ]} (C\epsilon, \mathcal{F}, L_2(P)) & \leq e^{4\frac{M}{\epsilon}\log2} (1+\frac{4}{\epsilon})^d \\
	& =  e^{ 4\log2({\frac{1}{\epsilon} })^{1+\frac{2}{\delta}} }(1+\frac{4}{\epsilon})^d .
\end{align*}
Therefore
\[	\mathcal{N}_{ [ \ ]} (\epsilon, \mathcal{F}, L_2(P)) \leq e^{ 4\log2({\frac{C}{\epsilon} })^{1+\frac{2}{\delta}} }(1+\frac{4C}{\epsilon})^d.\]

Then for the bracketing entropy integral of $\mathcal{F}$, we have
\begin{align*}
	J_{[ \ ]}(1, \mathcal{F}, L_2(P)) &= \int_{0}^{1} \sqrt{ \log \mathcal{N}_{[ \ ]}(\epsilon, \mathcal{F}  , L_2(P)) }d\epsilon \\
	&\leq \int_{0}^{1} \sqrt{ 4\log2 ({\frac{C}{\epsilon} })^{1+\frac{2}{\delta}}  + d\log(1+\frac{4C}{\epsilon})   } d\epsilon\\
	& \leq \int_{0}^{1}  \left(  2\sqrt{\log2}({\frac{C}{\epsilon} })^{\frac{1}{2}+\frac{1}{\delta}} + \sqrt{d\log(1+\frac{4C}{\epsilon}) }   \right) d\epsilon .
\end{align*}
The second inequality follows from $\sqrt{a+b} \leq \sqrt{a}+\sqrt{b}$ ($a, b \geq 0$). Observe that when  $\frac{1}{2} + \frac{1}{\delta} < 1$, i.e. $\delta > 2$, $J_{[ \ ]}(1, \mathcal{F}, L_2(P)) < \infty$. Thus when the $(4+\delta)$-th ($\delta > 0$) moment of $P$ is finite, $\mathcal{F}$ is $P$-Donsker. 
\end{proof}

\begin{remark}
Note that for $I_1$, we have a tighter estimation.
\begin{align*}
	I_1 &= \int_{ \{ |\theta^T x| > M \} } \left(f(\theta^T x)- g(\theta^T x) \right)^2 dP(x) \\
	&= \int_{ \{ |y| > M \} } (f(y)-g(y))^2 d{\theta^{*}_{\#}}P(y)\\
	& \leq \int_{ \{ |y| > M \} } 4|y|^2 d{\theta^{*}_{\#}}P(y).
\end{align*}
Then the condition becomes $\sup\limits_{\theta \in S^{d-1}} \mathbb{E}|\theta^T X|^{4+\delta} < \infty$, but it's equivalent to $\mathbb{E} \|X\|^{4+\delta} < \infty$.
\end{remark}

\subsection{Proof of Corollary 7}
\begin{proof} 

Let $\mathcal{F}_1$ be a $\epsilon$-bracket cover of $\mathcal{F}|_{[0,M]}$ (in the sense of $\| \cdot \|_{\infty}$) with $\log | \mathcal{F}_1 |=c  \frac{M^{\alpha}}{\epsilon^{\beta}}$. Then for any $\epsilon$-bracket $[g,f] \in \mathcal{F}_1$, let $f(x)=-LM+f(M)+Lx, \ g(x) = LM+g(M)-Lx$ when $x > M$ and $f(x)=-LM+f(-M)-Lx, \ g(x) = LM+g(-M)+Lx$ when $x < -M$. 

Similarly let $\mathcal{G}$ be a minimal $\epsilon$-cover of $S^{d-1}$ in the sense of Euclidean norm that is $|\mathcal{G}|=\mathcal{N}(\epsilon, S^{d-1}, \|\cdot \|)$. Then $\mathcal{H}:= \{[g(\theta^T x)-L\epsilon\|x||, f(\theta^T x)+L\epsilon\|x\|] : [g,f] \in \mathcal{F}_1, \ \theta \in \mathcal{G} \}$ is a $C\epsilon$-bracket cover of $\widehat{\mathcal{F}}$ where $C$ depends only on the moment of $P$ and $L$. 

From the same estimation as Theorem 2, let $M^{\delta} = \frac{1}{\epsilon^2}$ (provided the ($2+\delta$)-th moment of $P$ is finite), we have
\begin{align*}
& \|(f(\theta^T x)+L\epsilon\|x\|_2) - (g(\theta^T x)-L\epsilon\|x||_2) \|_{L^2(P)} \\
& \leq 2\epsilon^2 + 8L^2 m_2 \epsilon^2 + 8L^2m_{2+\delta} \epsilon^2 := C^2 \epsilon^2.
\end{align*} 

Thus we have 
\[ \log \mathcal{N}_{ [ \ ]}(\epsilon, \widehat{F}, L_2(P)) \leq c(\frac{C}{\epsilon})^{\frac{2\alpha}{\delta}+\beta} + d\log(1+\frac{4C}{\epsilon}).\]

Then the bracketing entropy integral of $\widehat{\mathcal{F}}$ is finite if $\frac{\alpha}{\delta}+\frac{\beta}{2} < 1$ i.e. $\delta > \frac{2\alpha}{2-\beta}$. In this condition, $\widehat{\mathcal{F}}$ is $P$-Donsker.
\end{proof}

\subsection{Proof of Theorem 3}
\begin{proof}
Note that
\begin{align*}
&\mathbb{E}[MSW_1(\mu_n, \mu)] \\
&= \mathbb{E}[\sup_{f\in\mathcal{F}} \int f(x)d\mu_n(x)-\int f(x)d\mu(x) ]\\
&=\mathbb{E}[\sup_{f\in\mathcal{F}} \int f(x)I_{\|x\|\leq M}d\mu_n(x)-\int f(x)I_{\|x\|\leq M}d\mu(x)+ \int f(x)I_{\|x\|> M}d\mu_n(x)-\int f(x)I_{\|x\|> M}d\mu(x)]\\
&\leq  \mathbb{E}[\sup_{f\in\mathcal{F}} \int f(x)I_{\|x\|\leq M}d\mu_n(x)-\int f(x)I_{\|x\|\leq M}d\mu(x)+ \sup_{f\in \mathcal{F}} \int f(x)I_{\|x\|> M}d\mu_n(x)-\int f(x)I_{\|x\|> M}d\mu(x)]\\
&= \mathbb{E}[\sup_{f\in\mathcal{F}_M} \int f(x)d\mu_n(x)-\int f(x)d\mu(x)]+ \mathbb{E}[\sup_{f\in \mathcal{F}_{M^c}} \int f(x)d\mu_n(x)-\int f(x)d\mu(x)],
\end{align*}
where $M$ is constant to be determined and $\mathcal{F}_M:=\{f(x)I_{\|x\|\leq M}:\ f \in \mathcal{F}\}$, 
$\mathcal{F}_{M^c}:=\{f(x)I_{\|x\|>M}: \ f \in \mathcal{F}\}$. And let $I_1$ be the first term above and $I_2$ be the last term.

By the standard symmetrization and Theorem 16 in \cite{67}, we have
\begin{equation}
I_1 \lesssim \mathbb{E}\left[\inf_{\delta>0} \left\{ \delta+\frac{1}{\sqrt{n}}\int_{\delta}^{\sigma_n} \sqrt{\log \mathcal{N}(\mathcal{F}_M, \|\cdot\|_{L^2(\mu_n)}, \epsilon)}d\epsilon\right\} \right], 
\end{equation}
where $\sigma_n=\sup_{f\in \mathcal{F}_M} \|f\|_{L^2(\mu_n)}$.

Recall that 
\[
\mathcal{F}:= \{f :\mathbb{R}^d\to \mathbb{R} \ | \ f(x)=g(\theta^T x), g \in Lip_{1,0}(\mathbb{R}), \theta \in S^{d-1} \}.\]

Notice that for any $g_1, g_2 \in Lip_{1,0}([-M, M]) ,\theta_1,\theta_2 \in S^{d-1}$ and $\|x\|\leq M$, we have	
\begin{align*}
|g_1(\theta_1^T x)-g_2(\theta_2^T x)| &\leq |g_1(\theta_1^T x)-g_2(\theta_1^T x)|+|g_2(\theta_1^T x)-g_2(\theta_2^T x) |\\
&\leq \|g_1-g_2\|_{\infty}+\|\theta_1-\theta_2\|M.
\end{align*}

Therefore
\begin{align*}
\mathcal{N}(\mathcal{F}_M, \|\cdot\|_{\infty}, \epsilon) &\leq 
\mathcal{N}(Lip_{1,0}([-M, M]), \|\cdot\|_{\infty}, \frac{\epsilon}{2}) \cdot \mathcal{N}(S^{d-1}, \|\cdot\|_2, \frac{\epsilon}{2M}) \\
&\leq e^{c\frac{M}{\epsilon}}(1+c\frac{M}{\epsilon})^d,	
\end{align*}	
where $c$ is a universal constant. 

Then by (4) and the fact that $\|f\|_{L^2(\mu_n)}\leq \|f\|_{\infty}$, we obtain that
\begin{align*}
I_1 &\lesssim \mathbb{E}\left[\inf_{\delta>0} \left\{ \delta+\frac{1}{\sqrt{n}}\int_{\delta}^{\sigma_n} \sqrt{\log \mathcal{N}(\mathcal{F}_M, \|\cdot\|_{\infty}, \epsilon)}d\epsilon\right\} \right] \\
&\lesssim \mathbb{E}\left[ \inf_{\delta>0} \left\{ \delta+\frac{1}{\sqrt{n}}\int_{\delta}^{\sigma_n}\sqrt{\frac{M}{\epsilon}}d\epsilon \right\} \right]\\
&=\mathbb{E}\left[ \inf_{\delta>0} \left\{ \delta+\frac{1}{\sqrt{n}}2\sqrt{M}(\sqrt{\sigma_n}-\sqrt{\delta}) \right\} \right].
\end{align*}
By choosing $\delta=\frac{\sigma_n}{\sqrt{n}}$, we have
\begin{align*}
I_1 &\lesssim 
\frac{1}{\sqrt{n}}\mathbb{E}[\sigma_n]+\frac{1}{\sqrt{n}}\sqrt{M}\mathbb{E}[\sqrt{\sigma_n}].
\end{align*}

For $\sigma_n$, we have
\begin{align*}
\mathbb{E}[\sigma_n] &=\mathbb{E}\left[\sup_{f\in \mathcal{F}_M} \sqrt{\frac{1}{n}\sum\limits_{i=1}^n f(X_i)^2} \right]\\
&\leq \mathbb{E}\left[ \sqrt{ \frac{1}{n}\sum\limits_{i=1}^n \|X_i\|^2 I_{\|X_i\|\leq M}} \right] \\
&\leq \sqrt{\mathbb{E}[\frac{1}{n}\sum\limits_{i=1}^n \|X_i\|^2 I_{\|X_i\|\leq M}]}\\
&\leq \sqrt{\mathbb{E}[\|X\|^2]},
\end{align*}
where the second inequality follows by the fact that $t(x)=x^{\frac{1}{2}}$ is concave and Jensen inequality. Similarly, $\mathbb{E}[\sqrt{\sigma_n}] \leq \sqrt[4]{\mathbb{E}\|X\|^2}$. 

Therefore under conditions (2), we have $I_1 \lesssim \sqrt{\frac{M}{n}}$ and
 
\[I_2 \leq 2\mathbb{E}[\|X\|I_{\|X\|>M}] \leq 2\frac{\mathbb{E}\|X\|^{1+\delta}}{M^{\delta}} .\]
Then 
\begin{align*}
\mathbb{E}[MSW_1(\mu_n, \mu)] &\leq I_1+I_2 \\
&\lesssim \sqrt{\frac{M}{n}}+\frac{\mathbb{E}\|X\|^{1+\delta}}{M^{\delta}}\\
&\lesssim n^{-\frac{\delta}{1+2\delta}}=n^{-\frac{p-1}{2p-1}},
\end{align*}
where the last inequality follows by choosing $M=n^{\frac{1}{1+2\delta}}$.

Under condition (3), the result can be derived straightly from Theorem 3.5.13 in \cite{70} and applying Markov inequality for the second term in the right of (3.203).

Under condition (1), $\mathbb{E}[\|X\|^2]<\infty$ may not hold, so the method above doesn't work. But we can utilize Dudley's integral inequality directly (Corollary 5.25 in \cite{68}). 

By the standard symmetrization argument, we have 
\[I_1 \leq 2\mathbb{E}_{X} \mathbb{E}_{\epsilon} \sup_{f\in\mathcal{F}_{M}}\frac{1}{n}\sum\limits_{i=1}^{n} \epsilon_i f(X_i).  \] 
With $X_1,\cdots, X_n$ fixed, we define stochastic process indexed by $\mathcal{F}_M$ as 
\[X_f= \frac{1}{n}\sum\limits_{i=1}^{n} \epsilon_i f(X_i).\]
To invoke Dudley's integral inequality, we  need to verify that the random process $(X_f)_{f \in \mathcal{F}_M}$ is sub-Gaussian. By Hoeffding's inequality, we have 
\begin{align*}
\mathbb{E}_{\epsilon}[e^{\lambda(X_f-X_g)}] &= \prod_{i=1}^n \mathbb{E}_{\epsilon}[e^{\frac{\lambda \epsilon_i(f(X_i)-g(X_i))}{n}}] \\
&\leq \prod_{i=1}^n e^{\frac{\lambda^2(f(X_i)-g(X_i))^2}{2n^2}} \\
&\leq e^{\frac{\lambda^2 \|f-g\|_{\infty}^2}{2n}}.
\end{align*}
Therefore, Dudley's integral inequality yields that
\begin{align*}
I_1 &\lesssim \mathbb{E}_{X}	\mathbb{E}_{\epsilon} \sup_{f\in\mathcal{F}_{M}}\frac{1}{n}\sum\limits_{i=1}^{n} \epsilon_i f(X_i) \\
&\lesssim \frac{1}{\sqrt{n}} \int_{0}^{\infty}  \sqrt{\log\mathcal{N}(\mathcal{F}_M, \|\cdot\|_{\infty}, \epsilon)d\epsilon } \\
&\lesssim \frac{1}{\sqrt{n}} \int_{0}^{M} \sqrt{\frac{M}{\epsilon}}d\epsilon \\
&\lesssim \frac{M}{\sqrt{n}}.
\end{align*}	
Therefore
\begin{align*}
\mathbb{E}[MSW_1(\mu_n, \mu)] &\leq I_1+I_2 \\
&\lesssim \frac{M}{\sqrt{n}}+ \frac{1}{M^{\delta}}\mathbb{E}[\|X\|^{1+\delta}]\\
&\lesssim n^{-\frac{\delta}{2(1+\delta)}} =n^{-\frac{p-1}{2p}},
\end{align*}
where the last inequality follows by choosing $M=n^{-\frac{1}{2(2+\delta)}}$.
\end{proof}

\subsection{Proof of Proposition 5}
\begin{proof}
By taking $MSW_1(\mu_n, \mu)$ as a supremum of an empirical process, the concentration inequalities follows from a direct application of Lemma 1 and Theorem 4 in \cite{26}. For brevity, We omit the details.

Assume that $\mu$ is induced by a $\sigma^2$-Subgaussian random vector $X$ with value in $\mathbb{R}^d$. Let $g(x_1,\cdots, x_n)=\sup_{f\in \mathcal{F}}|\sum\limits_{i=1}^n\frac{f(x_i)-\mathbb{E}f(X)}{n}|$. It's easy to verify that $g$ is $1/n$-Lipschitz continuous with respect to the distance $d((x_1,\cdots, x_n), (x_1^{'}, \cdots, x_n^{'})):=\sum\limits_{i=1}^n \|x_i-x_i^{'}\|$. Then 
\[MSW_1(\mu_n,\mu)=\sup_{f\in \mathcal{F}}|\sum\limits_{i=1}^n\frac{f(X_i)-\mathbb{E}f(X)}{n}|=g(X_1,\cdots, X_n).\]

To invoke transpotation cost inequality, we need to verify that $\mu$ satisfies the transpotation cost inequality. For $X\sim \mu$, let $Y$ be an independent copy of $X$, for any $f \in Lip_1(\mathbb{R}^d)$, we have
\begin{align*}
	\mathbb{E} e^{\lambda (f(X)-\mathbb{E}f(X))} &= \mathbb{E}_X e^{\lambda (f(X)-\mathbb{E}_X f(X))}\\
	&= \mathbb{E}_X e^{\lambda (\mathbb{E}_Y f(X)-\mathbb{E}_Y f(Y))} \\
	&= \mathbb{E}_X e^{\lambda\mathbb{E}_Y(f(X)-f(Y))} \\
	&\leq \mathbb{E}_{X,Y} e^{\lambda(f(X)-f(Y))} \\
	&= \frac{1}{2} \mathbb{E}_{X,Y} e^{\lambda(f(X)-f(Y))}+\frac{1}{2} \mathbb{E}_{X,Y} e^{-\lambda(f(X)-f(Y))} \\
	&\leq \frac{1}{2} \mathbb{E}_{X,Y} e^{\lambda \|X-Y\|}+\frac{1}{2} \mathbb{E}_{X,Y} e^{-\lambda \|X-Y\|} \\
	&= \mathbb{E}_{X,Y,\epsilon} e^{\lambda \epsilon \|X-Y\|}.
\end{align*}	
By the fact that $X-Y$ is $2\sigma^2$-Subgaussian and Proposition 10 in \cite{66}, we can infer that $\epsilon \|X-Y\|$ is $16d\sigma^2$-Subgaussian, which implys that for any $f \in Lip_1(\mathbb{R}^d)$, $f(X)$ is $16d\sigma^2$-Sungaussian. Then from Bobkov Theorem (Theorem 4.8 in \cite{68}), the following inequality holds.
\[W_1(\mu, \nu)\leq \sqrt{2(16d\sigma^2)D(\nu\|\mu)}, \ for \ \forall \ \nu \in \mathcal{P}(\mathbb{R}^d).\]
By tensorization, we have
\[W_1(\mu^{\otimes n}, \nu) \leq \sqrt{2(16nd\sigma^2)D(\nu\|\mu)}, \ for \ \forall \ \nu \in \mathcal{P}(\mathbb{R}^{nd}).\]	
Since $g$ is $1/n$-Lipschitz continuous with respect to $d$, by Bobkov Theorem we have 
\[P(|MSW_1(\mu_n, \mu)-\mathbb{E}MSW_1(\mu_n, \mu)|\geq t) \leq 2e^{-\frac{nt^2}{32d\sigma^2}}.\]

And for $SW_1$, we have
\begin{align*}
	&\mathbb{E}e^{\lambda(SW_1(\mu,\mu_n)-\mathbb{E}SW_1(\mu,\mu_n))} \\
	&=\mathbb{E} e^{\lambda \int_{S^{d-1}}  (W_1(\theta_{\#}^{*}\mu,\theta_{\#}^{*}\mu_n )-\mathbb{E}W_1(\theta_{\#}^{*}\mu,\theta_{\#}^{*}\mu_n ) )d\sigma(\theta) } \\
	&\leq \mathbb{E}[\int_{S^{d-1}} e^{\lambda (W_1(\theta_{\#}^{*}\mu,\theta_{\#}^{*}\mu_n )-\mathbb{E}W_1(\theta_{\#}^{*}\mu,\theta_{\#}^{*}\mu_n)} d\sigma(\theta)] \\
	&=\int_{S^{d-1}} \mathbb{E}[e^{\lambda (W_1(\theta_{\#}^{*}\mu,\theta_{\#}^{*}\mu_n )-\mathbb{E}W_1(\theta_{\#}^{*}\mu,\theta_{\#}^{*}\mu_n))}]d\sigma(\theta)\\
\end{align*}
Similarily, let $h(y_1, \cdots, y_n)=\sup\limits_{f\in Lip_1(\mathbb{R})} \sum\limits_{i=1}^n |\frac{f(y_i)-\mathbb{E}f(\theta^T X_i)}{n}|$, then $h$ is $1/n$-Lipschitz continuous with respect to the distance	 $d((y_1,\cdots, y_n), (y_1^{'}, \cdots, y_n^{'})):=\sum\limits_{i=1}^n |y_i-y_i^{'}|$. Then we have $W_1(\theta_{\#}^{*}\mu,\theta_{\#}^{*}\mu_n)=h(\theta^T X_1, \cdots, \theta^T X_n).$

By the fact that univariate random variable $\theta^T(X-Y)$ is symmetric, we can infer that $\epsilon|\theta^T(X-Y) |$ has the common distribution with $\epsilon(\theta^T(X-Y) )$, which implys that $\epsilon|\theta^T(X-Y) |$ is $2\sigma^2$-Subgaussian. By tensorization, we have
\[W_1(({\theta_{\#}^{*}\mu})^{\otimes n}, v)\leq \sqrt{2(2n\sigma^2)D(\nu\|({\theta_{\#}^{*}\mu})^{\otimes n} )}, \ for \ \forall \nu \in \mathcal{P}(\mathbb{R}^n).\]

Then 	
\[\mathbb{E}e^{\lambda(SW_1(\mu,\mu_n)-\mathbb{E}SW_1(\mu,\mu_n))}
\leq \int_{S^{d-1}} e^{-\frac{\lambda^2}{2}\frac{2\sigma^2}{n} }d\sigma(\theta) =e^{-\frac{\lambda^2}{2}\frac{2\sigma^2}{n} }.\]
Thus, 
\[P(|SW_1(\mu_n, \mu) -\mathbb{E}SW_1(\mu_n,\mu)|\geq t) \leq 2e^{-\frac{nt^2}{4\sigma^2}}.\]

\end{proof}

\subsection{Proof of Corollary 10}
\begin{proof}
(i) Theorem 2 implies that $\mathcal{F}$ is a Donsker class under both $P$ and $Q$, then the 
conclusion follows directly from Chapter 3.7 of \cite{15}. 

(ii) It follows by Theorem 3.7.7 in \cite{15}.
\end{proof}

\subsection{Proof of Proposition 6}
\begin{proof}
Under (1), the proof is almost same as Theorem 2.3 in \cite{69}. The only thing we need to prove is that for each $\epsilon>0$ and each $\delta>0$ we can choose some $r \in \mathbb{N}$ such that
\[\mathbb{E} \int_{S^{d-1}\times [-r,r]^c} |P_m^B-Q_n^B|dtd\sigma(\theta)<\epsilon \delta  .\]
It's easy to obtain a similar form as equation (8) in \cite{69}.
\begin{align*}
&\ \ \mathbb{E} \int_{S^{d-1}\times [-r,r]^c} |P_m^B-Q_n^B|dtd\sigma(\theta)\\
&\leq \left(\mathbb{E} \int_{S^{d-1}\times [-r,r]^c} |P_m^B-Q_n^B|^2dtd\sigma(\theta) \right)^{1/2} \\
&= \left( \int_{S^{d-1}\times [-r,r]^c} \mathbb{E}|P_m^B-Q_n^B|^2dtd\sigma(\theta) \right)^{1/2} \\
&\leq \sqrt{ \frac{m}{N} \frac{1}{m} \sum\limits_{i=1}^m \|X_i\|I_{ \{\|X_i\|>r\}} + \frac{n}{N} \frac{1}{n} \sum\limits_{k=1}^n \|Y_k\|I_{ \{\|Y_k\|>r\}}  },
\end{align*}
where the first inequality follows from Cauchy inequality and the second inequality follows from (8) in \cite{69}. Given almost every $X_1, \cdots, X_m, Y_1, \cdots, Y_n$, the result follows from the law of large numbers. Then by the same argument as Theorem 2.3 in \cite{69}, we conclude.  

Under (ii), it follows immediately from Corollary 8.
\end{proof}

\bibliographystyle{IEEEtran}
\bibliography{CLT.bib}

\end{document}